\documentclass[10pt]{amsart}

\usepackage[utf8]{inputenc}

\usepackage{amsmath}
\usepackage{amssymb}
\usepackage{setspace}

\usepackage{amsthm}
\usepackage{fullpage}
\usepackage{biblatex}
\usepackage{breqn}
\newtheorem{theorem}{Theorem}[]
\newtheorem{lemma}{Lemma}[]

\newtheorem*{remark}{Remark}{}

\newtheorem{proposition}{Proposition}[section]

\title{A lower bound on high moments of character sums}
\author{Barnabás Szabó }
\address{Mathematics Institute, Zeeman Building, University of Warwick, Coventry CV4 7AL, England}
\email{\tt Barnabas.Szabo@warwick.ac.uk}

\begin{document}

\begin{abstract}
    For any real $k\geq 2$ and large prime $q$, we prove a lower bound on the $2k$-th moment of the Dirichlet character sum
    \begin{equation*}
    \frac{1}{\phi(q)} \sum_{\substack{\chi \text{ mod }q\\ \chi\neq \chi_0}} \Big| \sum_{n\leq x} \chi(n)\Big|^{2k},
\end{equation*}
where $1\leq x\leq q$, and $\chi$ is summed over the set of non-trivial Dirichlet characters mod $q$. Our bound is known to be optimal up to a constant factor under the Generalised Riemann Hypothesis. We also get a sharp lower bound on moments of theta functions using the same method.
\end{abstract}
\maketitle
% Bounding moments of central values of  families of $L$-functions has attracted a lot of attention recently after Keating, Snaith () have formulated precise conjectures on what these values are asymptotically using random matrix theory. Since proving these conjectures is out of reach with our current technology, it makes sense to look for reasonable upper and lower bounds instead. 
\section{Introduction}
In a short but very elegant paper \cite{rudnick2005lower}, Rudnick and Soundararajan developed a method to give lower bounds for moments of families of $L$-functions, which are believed to be optimal up to constants. In particular, they show that if $k$ is a natural number and $q$ is a prime, then
\begin{equation*}
    \frac{1}{\phi(q) }\sum_{\substack{\chi \text{ mod }q\\ \chi\neq \chi_0}} |L(1/2, \chi)|^{2k}\gg_k (\log q)^{k^2},
\end{equation*}
where $\chi$ runs through the set of non-trivial Dirichlet characters mod $q$, and $L(s,\chi)$ is the Dirichlet $L$-function associated to $\chi$. The method of proof can be described as follows. For each character $\chi$ mod $q$, one finds a suitable proxy object $A(\chi)\geq 0$, and applies Hölder's inequality in the form (assuming $k\geq 2$)
\begin{equation*}
    \Big(\sum_{\substack{\chi \text{ mod }q\\ \chi\neq \chi_0}} |L(1/2, \chi)|^{2k}\Big)^{1/k}\cdot\Big(\sum_{\substack{\chi \text{ mod }q\\ \chi\neq \chi_0}} A(\chi)^{\frac{k}{k-1} }\Big)^{\frac{k-1}{k}}\geq \sum_{\substack{\chi \text{ mod }q\\ \chi\neq \chi_0}} |L(1/2,\chi)|^2A(\chi).
\end{equation*}
If we choose $A(\chi)$, so that we are able to calculate the second and the third sum in this expression, then we certainly get a lower bound on the $2k$-th moment of $L(1/2,\chi)$, however our choice of $A(\chi)$ should be such that the lower bound we get is close to optimal. The case of equality in Hölder's inequality tells us that we should choose $A(\chi)$ to be such that it behaves like $|L(1/2,\chi)|^{2(k-1)}$. Rudnick and Soundararajan choose $ A(\chi)=\big|\sum_{n\leq x}\frac{\chi(n)}{n^{1/2}}\big|^{2(k-1)}$ with $x=q^{\frac{1}{2k}}$, which is reasonable since $L(1/2,\chi)=\sum_{n=1}^{\infty} \frac{\chi(n)}{n^{1/2} }$. 

We use the proof method described above to give a lower bound on sufficiently high moments of unweighted character sums. Our theorem reads as follows.
\begin{theorem}
\label{t1}
    Let $k\geq 2$ be a real number, $q$ a large prime. If $1\leq x\leq q^{1/2}$, then
\begin{equation*}
    \frac{1}{\phi(q)}\sum_{\substack{\chi \text{ mod }q\\ \chi\neq \chi_0}} \bigg| \sum_{n\leq x} \chi(n)\bigg|^{2k}\gg_k x^k (\log x)^{(k-1)^2}. 
\end{equation*}
If $q^{1/2}\leq x\leq q/2$, then
\begin{equation*}
    \frac{1}{\phi(q)} \sum_{\substack{\chi \text{ mod }q\\ \chi\neq \chi_0}} \bigg| \sum_{n\leq x} \chi(n)\bigg|^{2k}\gg_k x^k \Big(\log \frac{q}{x}\Big)^{(k-1)^2}.
\end{equation*}
\end{theorem}
Theorem \ref{t1} has been proven in the special case when $k\in \mathbb{N}$ and $x=q^{\alpha}$ for a fixed $0<\alpha<1$ by Munsch and Shparlinski  (see Theorem 1.6. in \cite{munsch2016upper}), so our main contribution is that we are able to handle non-integer moments in a wider range.
Note that the bounds in Theorem \ref{t1} have a slightly different shape according to whether $x\leq q^{1/2}$ or $x\geq q^{1/2}$, which comes from the fact that the quantities $\sum_{n\leq x}\chi(n)$ and $\sum_{n\leq q/x}\chi(n)$ are closely related by a reciprocity relation. When $q/2\leq x\leq q$, a similar type of lower bound follows easily from Theorem \ref{t1}, since in this range the orthogonality relation and periodicity imply $\big| \sum_{n\leq x}\chi(n)\big|=\big|\sum_{x<n< q}\chi(n)\big|=\big|\sum_{ n<q-x}\chi(n)\big|$, where $0\leq q-x\leq q/2$ so Theorem \ref{t1} is applicable (we did not include this range in the statement for simplicity).

Theta functions are closely related to character sums. For a Dirichlet character $\chi$ mod $q$ define $\kappa=\kappa(\chi)=(1-\chi(-1))/2$, i.e. $\kappa=1$ if $\chi$ is odd and $\kappa=0$ if $\chi$ is even. The $\theta$ function corresponding to $\chi$ is defined as 
\begin{equation*}
\theta(x,\chi)=\sum_{n=1}^{\infty} \chi(n) n^{\kappa} e^{-\pi n^2 x/q}.
\end{equation*}
Modifying the proof of Theorem \ref{t1}, we will show the following.
\begin{theorem}
\label{t2}
    Let $X_q^+$ be the set of even characters mod $q$,  let $k\geq 2$ be a real number. We have
 \begin{equation*}
    \frac{1}{\phi(q)}\sum_{\chi \in X_q^+\backslash \{ \chi_0\} } |\theta(1,\chi)|^{2k} \gg_k q^{k/2}(\log q)^{(k-1)^2}.
\end{equation*}
Let $X_q^-$ denote the set of odd characters mod $q$, let $k\geq 2$ be a real number. We have
\begin{equation*}
    \frac{1}{\phi(q)}\sum_{\chi \in X_q^- } |\theta(1,\chi)|^{2k} \gg_k q^{3k/2}(\log q)^{(k-1)^2}.
\end{equation*}
\end{theorem}
Notice that for an even $\chi$, $\theta(1,\chi)=\sum_{n=1}^{\infty}\chi(n)e^{-\pi n^2/q}$ behaves similarly to $\sum_{n\leq q^{1/2}} \chi(n)$, since the weight $e^{-\pi n^2/q}$ starts decaying roughly at $n\approx q^{1/2}$, therefore it is not surprising that the lower bound has the same shape as Theorem \ref{t1} with $x=q^{1/2}$ (for odd $\chi$ the theta function behaves like $\sum_{n\leq q^{1/2}}n\chi(n)$ so we have an extra $q^k$ term in the lower bound). In fact the proof of Theorem \ref{t2} is essentially the same as that of Theorem \ref{t1}, except for one part of the argument where we need to alter the technicalities of the proof due the fact that the weights $e^{-\pi n^2/q}$ are not multiplicative. We will explain these in the last section. Again, in \cite{munsch2016upper} Munsch and Shparlinski proved Theorem \ref{t2} when $k$ is an integer.

In \cite{szabo2023high}, the author showed the corresponding upper bounds for the $2k$-th moment when $k> 2$ assuming the Generalised Riemann Hypothesis (for any positive integer $q$). Note that showing sharp unconditional upper bounds for these moments when $k$ is large seems hopeless with our current knowledge, since only one `badly behaving' character could mess up the bound. For example, when $x=q^{1/5}$ and $k=10$, say, the Burgess bound, which is currently the best upper bound for single character sums, does not rule out the possibility that there is some $\chi$ mod $q$ for which $\big| \sum_{n\leq x} \chi(n)\big|=\lfloor x\rfloor$, so the contribution from this would be $\frac{1}{\phi(q)}\big|\sum_{n\leq q^{1/5} }\chi(n)\big|^{20}\gg q^3$, which is much bigger than the expected $q^{2+o(1)}$ bound for the moment. This phenomenon appears in other moment inequalities too. For example, we expect the $2k$-th moment of the $\zeta$ function to be 
\begin{equation*}
    \int_{T}^{2T}|\zeta(1/2+it)|^{2k}\asymp_{k} T(\log T)^{k^2},
\end{equation*}
when $T$ is large.
The lower bound part of this is known unconditionally for all real $k\geq 0$ (see \cite{heap2020lower}), whereas the upper bound is only known when $k\leq 4$ unconditionally (see \cite{heap2019sharp}) and for all $k\geq 0$ under the Riemann Hypothesis (see \cite{harper2013sharp}). Also note that the upper bound for arbitrarily large $k$ would imply the Lindelöf hypothesis namely that for any $\epsilon>0$ and $T\leq t \leq 2T$ one has $\zeta(t)\ll T^{\epsilon}$.

Let us go back to discussing the behaviour of 
\begin{equation*}
    \sum_{\substack{\chi \text{ mod }q\\ \chi\neq \chi_0}}\Big|  \sum_{n\leq x} \chi(n)\Big|^{2k}.
\end{equation*}
As described above, when $k\geq 2$ we understand the order of magnitude of this quantity, even if the upper bounds are conditional. We expect the moment to have the same behaviour when $1 \leq k\leq 2$ (this will be explained later after introducing random multiplicative functions), however our methods break down for $1< k<2$. It would certainly be interesting to prove some results (even conditionally on GRH) in this direction. When $0\leq k\leq 1$ the moment is now known to behave differently. In \cite{harper2023typical} Harper shows that if $q$ is a prime and $0\leq k\leq 1$, then 
\begin{equation*}
     \frac{1}{\phi(q)} \sum_{\substack{\chi \text{ mod }q\\ \chi\neq \chi_0}} \Big| \sum_{n\leq x} \chi(n) \Big|^{2k} \ll \bigg( \frac{x}{1+(1-k)\sqrt{\log\log \min(x, q/x)} }\bigg)^k, 
\end{equation*}
and he conjectures that this is indeed the right order of magnitude in this range.

Our proof techniques are heavily influenced by previous work of Harper (\cite{harper2020moments}, \cite{harper2013sharp}, \cite{harper2023typical}), particularly the paper \cite{harper2020moments} in which Harper determines the order of magnitude of high real moments ($k\geq 1$) of the Steinhaus random multiplicative function.  

We now define the Steinhaus random multiplicative function, which is a random function $f(n)$ defined on the integers, and in many aspects is a good model for character sums. Let $(f(p))_{p \text{ prime }}$ be a sequence of independent Steinhaus random variables (i.e. variables which are uniformly distributed on the complex unit circle $\{z: |z|=1\}$), and we make $f(n)$ completely multiplicative, that is if $n=p_1^{a_1}\cdots p_l^{a_l}$, where the $p_i$ are distinct primes, then $f(n)=f(p_1)^{a_1}\cdots f(p_l)^{a_l}$. One may expect that $f(n)$ simulates the behaviour of $\chi(n)$ for an `average $\chi$', since both are completely multiplicative, take values on the unit circle and the values of $\chi$ on primes do not influence each other too much, at least when we take primes up to $q$. In fact, we have
\begin{equation*}
    \mathbb{E}f(n)\overline{f(m)}=\mathbf{1}(n=m),
\end{equation*}
which is similar to the orthogonality relation for characters
\begin{equation*}
    \frac{1}{\phi(q)}\sum_{\chi \text{ mod }q}\chi(n)\bar{\chi}(m)=\mathbf{1}\big(n\equiv m\text{ (mod }q)\big),
\end{equation*}
In fact, if $1\leq n,m<q$, we have $\mathbf{1}\big(n\equiv m\text{ (mod }q)\big)=\mathbf{1}(n=m)$, so these two orthogonality relations have the same form. This will allow us to replace a polynomial expression in the $\chi$ with the corresponding polynomial expression for $f$ as long as the length of the polynomial is less than $q$ (see the beginning of Section 4). Generally, handling expressions involving the random multiplicative function is easier, since $f(p)$ are independent, whereas the $\chi(p)$ are not. Moreover random quantities allow us to directly use tools from probability theory.

Note that biggest discrepancy between these two objects is that a Dirichlet character is periodic modulo its conductor, on the other hand a random multiplicative function is almost surely non-periodic. However, when we consider initial sums $\{\sum_{n\leq x} \chi(n): \chi  \mod  q\}$, where $x\leq q^{1/2}$ is small enough, these sums can be modeled well by $\sum_{n\leq x} f(n)$. In particular, Harper \cite{harper2020moments} showed that for any real $k\geq 1$ and $x\geq 2$
\begin{equation*}
    \mathbb{E}\Big| \sum_{n\leq x} f(n) \Big|^{2k}\asymp_k x^{k}(\log x)^{(k-1)^2},
\end{equation*}
which is the same shape as the first part of our bound in Theorem \ref{t1}. Harper's moment bound therefore also suggests that Theorem \ref{t1} might hold for any real $k\geq 1$. Unfortunately our proof crucially uses $k\geq 2$.

Let us give a quick overview as to how the rest of the paper is organised. In Section 2 we state some results from the literature. In Section 3, assuming $x\leq q^{1/2}$, we define our proxy object $R(\chi)$, which is supposed to resemble the behaviour of $|\sum_{n\leq x} \chi(n)|^{2(k-1)}$. Note that the main idea in this paper is really the choice of $R(\chi)$, which will be much more complicated than the proxy object $A(\chi)$ defined on the first page, and the idea comes from a thorough understanding of the connection between Dirichlet characters and the Steinhaus random multiplicative function. In Section 3 we will also explain the motivation for the choice of $R(\chi)$. After applying Hölder's inequality, one needs a lower bound on the quantity $\frac{1}{\phi(q)} \sum_{\chi\neq \chi_0} \big|\sum_{n\leq x}\chi(n)\big|^2 R(\chi)$ and an upper bound on $\frac{1}{\phi(q)}\sum_{ \chi\neq \chi_0} R(\chi)^{\frac{k}{k-1}}$. These bounds will be proven in Section 4 and 5 respectively. In Section 6 we deal with the case $q^{1/2}\leq x\leq q/2$, in Section 7 we indicate the changes one has to make in the argument to prove Theorem \ref{t2}.

Finally let us mention the paper \cite{de2021small}, in which the authors give a lower bound for low moments ($0\leq 2k\leq 4/3$) of character sums which is sharp up to a small power of $\log q$.
\section{Background results}
\begin{lemma}
\label{eulerproduct}
    Let $f(n)$ be the Steinhaus random multiplicative function.
    Let  $\alpha,\beta,\sigma_1,\sigma_2\geq 0$ and $t_1,t_2$ be real numbers and $100(1+\max(\alpha^2,\beta^2))\leq z<y$. Then
    \begin{equation*}
    \begin{split}
        & \mathbb{E}\prod_{z\leq p\leq y}\Big|1-\frac{f(p)}{p^{1/2+\sigma_1+it_1}}\Big|^{-2\alpha}\Big|1-\frac{f(p)}{p^{1/2+\sigma_2+it_2}}\Big|^{-2\beta} \\
        =&\exp\bigg(\sum_{p\leq y}\frac{\alpha^2}{p^{1+2\sigma_1} }+\frac{\beta^2}{p^{1+2\sigma_2} }+\frac{2\alpha\beta \cos((t_2-t_1)\log p)}{p^{1+\sigma_1+\sigma_2 } }+O\big(\frac{\max(\alpha,\alpha^3,\beta,\beta^3)}{z^{1/2}}\big)\bigg).
    \end{split}
    \end{equation*}
\end{lemma}
\begin{proof}
    This follows from Euler product result 1 in \cite{harper2020moments}, with two small differences. One is that Harper assumes $\sigma_1=\sigma_2$ in the statement, however the same calculation with trivial modifications yields us the lemma. Secondly, Harper assumes $t_1=0$, however $f(p)$ and $f(p)p^{-it_1}$ have the same distribution, so if we substitute $\Tilde{f}(p)=f(p)p^{-it_1}$ for all primes $p$ and then make $\Tilde{f}$ totally multiplicative, then $\Tilde{f}$ is also a Steinhaus random multiplicative function, so we may apply the result to $\Tilde{f}$.
\end{proof}
\newpage
\begin{lemma}
\label{parseval}
    Let $(a_n)_{n\geq 1}$ be a sequence of complex numbers and $F(s)=\sum_{n=1}^{\infty}\frac{a_n}{n^s}$ be the corresponding Dirichlet series. If $\sigma_c$ denotes its abscissa of convergence, then for any $\sigma>\max(0,\sigma_c)$, we have
    \begin{equation*}
        \int_{1}^{\infty} \frac{\big|\sum_{n\leq x}a_n\big|^2 }{x^{1+2\sigma }}dx=\frac{1}{2\pi}\int_{-\infty}^{+\infty}\frac{|F(\sigma+it)|^2}{|\sigma+it|^2}dt.
    \end{equation*}
\end{lemma}
\begin{proof}
    This is Harmonic Analysis Result 1 in \cite{harper2020moments}, which refers to Theorem 5.4 of \cite{montgomery2007multiplicative}
\end{proof}
\begin{lemma}
\label{cosine}
Let $t\in \mathbb{R}$, $y\geq 2$. Then
    \begin{equation*}
\sum_{p\leq y} \frac{\cos(t \log p) }{p}\leq 
\begin{cases}
\log\log y+O(1)            & \text{if }  0\leq |t|\leq 1/\log y,        \\
\log(1/t)+O(1)        & \text{if }  1/\log y\leq |t| \leq 10,  \\
 \log\log |t|+O(1) & \text{if }  10\leq |t|.                \\
\end{cases}
\end{equation*}
\end{lemma}
\begin{proof}
    This is Lemma 2.9. in \cite{munsch2017shifted}.
\end{proof}
\begin{lemma}
\label{evenmoment}
    Let $a_n$ and $c_n$ be two complex sequences. Let $\mathcal{P}$ be a finite set of primes and define $\Tilde{d}(n)=\sum_{d|n} \mathbf{1}(p|d \implies p\in \mathcal{P})$. For any integer $j\geq 0$ we have
    \begin{equation*}
        \mathbb{E}\Big| \sum_{n\leq x} c_nf(n)\Big|^2 \Big|\sum_{p\in \mathcal{P} }\frac{a_pf(p)}{p^{1/2}}+\frac{a_{p^2} f(p^2)}{p}\Big|^{2j}\ll \Big( \sum_{n\leq x}\Tilde{d}(n)|c_n|^2\Big) \cdot (j!)\cdot \Big(2\cdot \sum_{p\in \mathcal{P}}\frac{|a_p|^2}{p}+\frac{6|a_{p^2} |^2}{p^2} \Big)^j
    \end{equation*}
\end{lemma}
\begin{proof}
    This is essentially Lemma 1 in \cite{harper2023typical}, however there is a slight difference we need to point out. There the lemma is stated using characters, i.e. the author upper bounds 
    \begin{equation*}
        \frac{1}{\phi(q)}\sum_{\chi\, \text{mod}\, q}\Big| \sum_{n\leq x} c_n\chi(n)\Big|^2 \Big|\sum_{p\in \mathcal{P} }\frac{a_p\chi(p)}{p^{1/2}}+\frac{a_{p^2} \chi(p^2)}{p}\Big|^{2j}
    \end{equation*}
under the condition $x(\max_{p\in \mathcal{P}} p)^{2j}<q$. This condition ensures that the Dirichlet polynomial $ \sum_{n\leq x} c_n\chi(n)\cdot \big( \sum_{p\in \mathcal{P} }\frac{a_p\chi(p)}{p^{1/2}}+\frac{a_{p^2} \chi(p^2)}{p}\big)^{j}$ has length less than $q$ so when applying the orthogonality relation, only the diagonal terms survive. In our case we have this by default, since $\mathbb{E}f(n)\overline{f(m)}=\mathbf{1}(n=m)$, so the same argument can be run.
\end{proof}
\section{Choosing the proxy object}

Let $x\leq q^{1/2}$.
In this section we define our proxy object $R(\chi)$ for each $\chi$ mod $q$, then apply Hölder's inequality and state the two bounds we need to show in order to prove Theorem 1 for $x\leq q^{1/2}$. The proof of these bounds will be given in the next two sections.

Firstly, let us give a heuristic explanation as to how we choose $R(\chi)$. As discussed in the introduction, we want Hölder's inequality to be close to optimal, so $R(\chi)$ should behave like $\big|\sum_{n\leq x}\chi(n)\big|^{2(k-1)}$, however we cannot simply choose $R(\chi)=\big|\sum_{n\leq x}\chi(n)\big|^{2(k-1)}$ as we would not be able to bound the arising quantities (unless $x$ is very small in terms of $q$). To proceed, we model $\big|\sum_{n\leq x}\chi(n)\big|^{2(k-1)}$ by $\big|\sum_{n\leq x}f(n)\big|^{2(k-1)}$, since we expect $f$ to behave like an `average' character $\chi$.

In fact, the main motivation comes from the work of Harper \cite{harper2020moments} on high moments of random multiplicative functions. Roughly speaking, he shows that (see equation (1.2) in \cite{harper2020moments})
\begin{equation*}
\begin{split}
    \mathbb{E}\big| \sum_{n\leq x} f(n)\big|^{2(k-1)} & \approx \frac{x^{k-1}}{(\log x)^{k-1} }  \mathbb{E}\Big( \int_{-1/2}^{1/2} \prod_{p\leq x} \Big|1-\frac{f(p)}{p^{1/2+it} }\Big|^{-2}dt\Big)^{k-1} \\
    & \approx  \frac{x^{k-1}}{(\log x)^{k-1} }  \mathbb{E}\Big( \frac{1}{\log x} \sum_{|l|\leq \log x/2} \prod_{p\leq x} \Big|1-\frac{f(p)}{p^{1/2+il/\log x} }\Big|^{-2} \Big)^{k-1},
\end{split}
\end{equation*}
where in the second line we discretized the integral noting that the Euler product does not change by much if we shift the imaginary part by $\frac{1}{\log x}$. As $k-1 \geq 1$ (i.e. we are in the high moments case), Harper observes (and later proves) that the main contribution to the expectation comes from a small (bounded) number of $l$-s for which $\prod_{p\leq x} \big|1-\frac{f(p)}{p^{1/2+il/\log x} }\big|^{-2}$ is large. Thus, roughly speaking we can write
\begin{equation*}
    \mathbb{E}\Big( \sum_{|l|\leq \log x/2} \prod_{p\leq x} \Big|1-\frac{f(p)}{p^{1/2+il/\log x} }\Big|^{-2} \Big)^{k-1}\approx \sum_{|l|\leq \log x/2} \mathbb{E} \prod_{p\leq x} \Big|1-\frac{f(p)}{p^{1/2+il/\log x} }\Big|^{-2(k-1)}.
\end{equation*}

In order to be able to handle the arising quantities, we will not take our Euler product up to $x$, but rather $y$, where $y$ is a small but fixed power of $x$, so our proxy object $R(\chi)$ will be something like
\begin{equation*}
     \sum_{|l|\leq \log y/2}\prod_{p\leq y} \Big|1-\frac{\chi(p)}{p^{1/2+il/\log y} }\Big|^{-2(k-1)}.
\end{equation*}
Unfortunately, working directly with this Euler product with our current knowledge is technically impossible (unless $y$ is very small compared to $q$), so we approximate it by a suitable Dirichlet polynomial in a way which has now become standard in bounding moments of $L$-functions.
% It turns out that the main contribution to the integral comes from a few large values of the product inside therefore 
% \begin{equation*}
%     \mathbb{E}\Big( \int_{-1/2}^{1/2} \prod_{p\leq x} \Big|1-\frac{f(p)}{p^{1/2+it} }\Big|^{-2} dt\Big)^{k-1}\approx \frac{1}{(\log x)^{k-2} }     \mathbb{E} \int_{-1/2}^{1/2} \prod_{p\leq x} \Big|1-\frac{f(p)}{p^{1/2+it} }\Big|^{-2(k-1)} dt
% \end{equation*}
% We will apply this heuristic to the character sum problem and choose $R(\chi)$ to resemble 
% \begin{equation*}
%     \int_{-1/2}^{1/2}\prod_{p\leq x}\bigg|1-\frac{\chi(p)}{p^{1/2+it} }\bigg|^{-2(k-1)}dt.
% \end{equation*}
 By taking the first two terms in the Taylor expansion of $\log\big(1-\frac{\chi(p)}{p^{1/2+it}}\big)$ and noting that the higher order terms contribute a factor of $O_k(1)$, we get
\begin{equation*}
    \prod_{p\leq y}\bigg|1-\frac{\chi(p)}{p^{1/2+it} }\bigg|^{-2(k-1)}\asymp_{k}  \exp\Big( 2(k-1)\Re \sum_{p\leq y}\frac{\chi(p)}{p^{1/2+it}}+\frac{\chi(p)^2}{2p^{1+2it}}\Big). 
\end{equation*}
In order for us to use the orthogonality relation for Dirichlet characters, we need to replace this exponential by a suitable short Dirichlet polynomial, whose length will be a small but fixed power of $q$. To do this we use the method developed by Harper in \cite{harper2013sharp}. Here the author partitions the interval $[1,y]$ as $1=y_0<y_1<\ldots <y_M=y$ for suitable parameters $y_m$, then for each interval $[y_{m-1}, y_m]$ one uses a truncated Taylor expansion 
\begin{equation*}
\exp\Big( 2(k-1)\Re \sum_{y_{m-1}<p\leq y_m}\frac{\chi(p)}{p^{1/2+it}}+\frac{\chi(p)^2}{2p^{1+2it}}\Big)\approx\bigg( \sum_{j=0}^{J_m} \frac{(k-1)^j}{j!}\Big( \Re \sum_{y_{m-1}<p\leq y_m}\frac{\chi(p)}{p^{1/2+it}}+\frac{\chi(p)^2}{2p^{1+2it}}\Big)^j\bigg)^2,
\end{equation*}
where the $J_m$ are suitable truncation parameters.

Let us now give the proper definition of $R(\chi)$.
We assume $x\leq q^{1/2}$. Let $y=x^{\frac{1}{C_0}}$, where $C_0$ is a sufficiently large absolute constant chosen at the end of the argument, which will depend on $k$. 

We define the subdivision $1=y_0<y_1<\ldots <y_M=y$ of the interval $[1,y]$ recursively by taking $y_M=y$, for any $2\leq m\leq M$ we have $y_{m-1}:=y_m^{1/20}$ and choose $M$ so that $y_1$ lies in $\big[y^{\frac{1}{20(\log \log y)^2}}, y^{\frac{1}{(\log \log y)^2}}\big]$.

We next define the truncation parameters $J_m$ for $1\leq m\leq M$. Let $J_1=(\log \log y)^{3/2}$, $J_M=\frac{C_0}{10^5k}$, and if $2\leq m\leq M-1$, we define $J_m=J_M+M-m$. Before we go on it is important to note the inequality
\begin{equation}
\label{shortpolynomial}
    \prod_{m=1}^M y_m^{10^4kJ_m}<x,
\end{equation}
which is an easy consequence of our definition. Our preliminary assumption on $C_0$ is that it is large enough that we have $J_M\geq \exp(10^4 k^2)$.

For any $\chi$ mod $q$, $1\leq m\leq M$ and $|l|\leq (\log y)/2$ integers we denote the Dirichlet polynomial
\begin{equation*}
    D_{m,l}(\chi):= \sum_{y_{m-1}< p\leq y_m } \frac{\chi(p)}{p^{1/2+il/\log y}} +\frac{\chi(p)^2}{2p^{1+2il/\log y}}. 
\end{equation*}
We define the truncated exponential as
\begin{equation*}
    R_{m,l}(\chi):= \bigg(\sum_{j=0 }^{J_m} \frac{(k-1)^j}{j!}\big(\Re D_{m,l}(\chi)\big)^j\bigg)^2,
\end{equation*}
and let our proxy object be
\begin{equation*}
     R(\chi):= \sum_{|l|\leq (\log y)/2} \prod_{m=1}^M  R_{m,l}(\chi).
\end{equation*}
The definition of $R(\chi)$ written out in full is
\begin{equation*}
    R(\chi)= \sum_{|l|\leq (\log y)/2} \prod_{m=1}^M  \bigg(\sum_{j=0}^{ J_m} \frac{(k-1)^j}{j!}\bigg(\Re \sum_{y_{m-1}< p\leq y_m } \frac{\chi(p)}{p^{1/2+il/\log y}} +\frac{\chi(p)^2}{2p^{1+2il/\log y}} \bigg)^j\bigg)^2. 
\end{equation*}
% \begin{remark}
% By squareroot cancellation, or an `average' $\chi$  we expect $D_{m,l}(\chi)$ 
% \end{remark}
Now that we have defined our proxy objects we apply Hölder's inequality with exponents $k$ and $\frac{k}{k-1}$ to get
\begin{equation}
\label{main_Hölder}
    \bigg(\frac{1}{\phi(q)}\sum_{\substack{\chi \text{ mod }q\\ \chi\neq \chi_0}}\bigg|\sum_{n\leq x} \chi(n)\bigg|^{2k}\bigg)^{1/k} \bigg( \frac{1}{\phi(q)}\sum_{\substack{\chi \text{ mod }q\\ \chi\neq \chi_0}} R(\chi)^{\frac{k}{k-1}}\bigg)^{(k-1)/k}\geq \frac{1}{\phi(q)} \sum_{\substack{\chi \text{ mod }q\\ \chi\neq \chi_0}} \bigg|\sum_{n\leq x}\chi(n)\bigg|^2 R(\chi).
\end{equation}
As for the lower bound part we will show the following.
\begin{proposition}
\label{p1}
There exists a bounded threshold (in terms of $k$), such that if $C_0$ exceeds that threshold we have 
\begin{equation*}
    \frac{1}{\phi(q)} \sum_{\substack{\chi \text{ mod }q\\ \chi\neq \chi_0}} \bigg|\sum_{n\leq x}\chi(n)\bigg|^2 R(\chi)\gg_k x(\log y)^{k^2-1}.
\end{equation*}
\end{proposition}
For the upper bound we prove the following.
\begin{proposition}
\label{p2}
   There exists a bounded threshold (in terms of $k$), such that if $C_0$ exceeds that threshold we have 
    \begin{equation*}
        \frac{1}{\phi(q)}\sum_{\substack{\chi \text{ mod }q\\ \chi\neq \chi_0}} R(\chi)^{\frac{k}{k-1}}\ll_k (\log y)^{k^2+1}
    \end{equation*}
\end{proposition}
Thus, for some bounded $C_0$, these two propositions together with \eqref{main_Hölder} imply 
\begin{equation*}
    \frac{1}{\phi(q)}\sum_{\substack{\chi \text{ mod }q\\ \chi\neq \chi_0}}\bigg|\sum_{n\leq x} \chi(n)\bigg|^{2k}\gg_k x(\log y)^{(k-1)^2},
\end{equation*}
which proves the first part of Theorem \ref{t1}, since $\frac{\log x}{\log y}=C_0\ll 1$.

Before moving on, we define the corresponding quantities for the random multiplicative function $f(n)$. We take
\begin{equation*}
    D_{m,l}(f)=\sum_{y_{m-1}< p\leq y_m } \frac{f(p)}{p^{1/2+il/\log y}} +\frac{f(p)^2}{2p^{1+2il/\log y}},
\end{equation*}
and define $R_{m,l}(f)$ and $R(f)$ the same way we defined $R_{m,l}(\chi)$ and $R(\chi)$. 
\section{The lower bound}
% \begin{itemize}
%     \item Step 1: Replace sum over primitive by all characters
%     \item Step 2: Replace sum over all characters by sum over random mult. functions
%     \item Step 3: State and prove the lower bound $e^{O(k^2)} x(\log y)^{k^2-1}$ we need for the exponential version of the sum 
%     \item Step 4: Prove the upper bound $e^{O(k^4)}x(\log y)^{k^2-1} e^{-J_m}$ for the $m$-th truncation.
%     \item Step 5: Show how these propositions imply Proposition 3.1.
% \end{itemize}
This section is devoted to the proof of Proposition \ref{p1}. We first go through the main steps of the argument.
As discussed in the introduction, one wants to work with random quantities instead of characters whenever possible.
\newpage
Hence our first step is to show that Proposition \ref{p1} follows from 
\begin{equation*}
    \mathbb{E}\bigg|\sum_{n\leq x} f(n)\bigg|^2 R(f)\gg_k x(\log y)^{k^2-1},
\end{equation*}
in other words we replace the averaging over characters by taking the expectation of the corresponding quantity involving the random multiplicative function. This step is relatively straightforward and it really depends on the fact that $R(\chi)$ is a sum of Dirichlet polynomials whose lengths are (much) shorter than $q$.

 Next, recall that
\begin{equation*}
    D_{m,l}(f)=\sum_{y_{m-1}< p\leq y_m } \frac{f(p)}{p^{1/2+il/\log y}} +\frac{f(p)^2}{2p^{1+2il/\log y}},
\end{equation*}
 and note that 
\begin{equation*}
    R(f)=\sum_{|l|\leq \log y/2} \prod_{m=1}^M \bigg(\sum_{j=0}^{J_m}\frac{(k-1)^j}{j!} \big(\Re D_{m,l}(f)\big)^j\bigg)^2
\end{equation*}
is a sum of truncated exponentials. Therefore we will replace $R(f)$ by the actual sum of exponentials
\begin{equation*}
    E(f):=\sum_{|l|\leq \log y/2}\exp\bigg(2(k-1)\Re\sum_{p\leq y}\frac{f(p)}{p^{1/2+il/\log y}} +\frac{f(p)^2}{2p^{1+2il/\log y}} \bigg),
\end{equation*}
and show that
\begin{equation*}
     \mathbb{E}\bigg|\sum_{n\leq x} f(n)\bigg|^2 E(f)\gg x(\log y)^{k^2-1}.
\end{equation*}

This is Proposition \ref{p3}, which is stated and proved later in this section.
Finally we need to check that in the process of replacing $R(f)$ by $E(f)$ the resulting error is small enough. This is probably the most involved part of the argument and will be proved in Proposition \ref{p4}. We note that even though the proof involves some lengthy calculations (mostly because the quantities we handle have `long' definitions), the argument is of elementary nature and we do not use any heavy machinery.

Now let us start the argument. As promised, the first step is to replace
\begin{equation*}
    \frac{1}{\phi(q)} \sum_{\chi \in \mathcal{X}_q^*} \bigg|\sum_{n\leq x}\chi(n)\bigg|^2 R(\chi)
\end{equation*}
with the corresponding quantity involving the Steinhaus random multiplicative function. In order to do this we need to to add $\big|\sum_{n\leq x}\chi_0(n)\big|^2R(\chi_0)$ to the expression so we can use orthogonality. We use the crude bound
\begin{equation*}
    D_{m,l}(\chi_0)=\sum_{y_{m-1}< p\leq y_m } \frac{\chi_0(p)}{p^{1/2+il/\log y}} +\frac{\chi_0(p)^2}{2p^{1+2il/\log y}}\leq y_m,
\end{equation*}
moreover by definition for all $1\leq m\leq M$ $y_m$ is larger then $100J_m$, say, therefore 
\begin{equation*}
    R_{m,l}(\chi_0)=\bigg(\sum_{j=0 }^{J_m} \frac{(k-1)^j}{j!}\big(\Re D_{m,l}(\chi_0)\big)^j\bigg)^2\leq (ky_m)^{2J_m}.
\end{equation*}
Hence by \eqref{shortpolynomial}, we have the bound
\begin{equation*}
    \Big|\sum_{n\leq x}\chi_0(n)\Big|^2R(\chi_0)\ll x^2(\log y)\prod_{m=1}^M (ky_m)^{2J_m} <x^{2+1/10},
\end{equation*}
so the contribution from the trivial character is negligible as $\frac{x^{2+1/10}}{\phi(q)}\leq x^{1/5}$, say (we assumed $x\leq q^{1/2}$). Therefore it is enough to show that
\begin{equation*}
   \frac{1}{\phi(q)} \sum_{\chi \text{ mod }q} \bigg|\sum_{n\leq x}\chi(n)\bigg|^2 R(\chi)\gg_k x(\log y)^{k^2-1}.
\end{equation*}
For any $\chi$ mod $q$, by expanding the brackets in the definition of $R(\chi)$ we may write
\begin{equation}
\label{expand}
    \bigg|\sum_{n\leq x}\chi(n)\bigg|^2 R(\chi)=\sum_{|l|\leq (\log y)/2}\sum_{\substack{n_1,n_2\leq N \\ m_1,m_2\leq x}} a_{n_1,n_2}\frac{\chi(m_1n_1)\bar{\chi}(m_2n_2)}{n_1^{1/2+il/\log y}n_2^{1/2-il/\log y} },
\end{equation}
for some complex coefficients $a_{n_1,n_2}$, and most importantly
\begin{equation*}
    xN\leq x \prod_{m=1}^M y_m^{4J_m}\leq x^{3/2}<q.
\end{equation*}
So when we average \eqref{expand} over $\chi$ mod $q$ the only non-zero contribution comes from when $m_1n_1=m_2n_2$ (diagonal terms), thus the orthogonality relation gives us
\begin{equation*}
    \frac{1}{\phi(q)}\sum_{\chi\, \text{mod}\, q}\bigg|\sum_{n\leq x}\chi(n)\bigg|^2 R(\chi)=\mathbb{E}\sum_{|l|\leq \log y/2}\sum_{\substack{n_1,n_2\leq N \\ m_1,m_2\leq x}} a_{n_1,n_2}\frac{f(m_1n_1)\bar{f}(m_2n_2)}{n_1^{1/2+il/\log y}n_2^{1/2-il/\log y} }=\mathbb{E}\bigg|\sum_{n\leq x} f(n)\bigg|^2 R(f),
\end{equation*}
so it is enough to show
\begin{equation*}
    \mathbb{E}\bigg|\sum_{n\leq x} f(n)\bigg|^2 R(f)\gg_k x(\log y)^{k^2-1}.
\end{equation*}
% \begin{equation*}
% \begin{split}
%     R(f)= &\sum_{|l|\leq \log y/2} \prod_{m=1}^M  R_{m,l}(f) \\
%             = &\sum_{|l|\leq \log y/2} \prod_{m=1}^M  \bigg(\sum_{j\leq J_m} \frac{(k-1)^j}{j!}\bigg(\Re \sum_{y_{m-1} p\leq y_m } \frac{f(p)}{p^{1/2+il/\log y}} +\frac{f(p)^2}{2p^{1+il/\log y}} \bigg)^j\bigg)^2. 
% \end{split}
% \end{equation*}
In order to prove this, we need the following two propositions. 
\begin{proposition}
\label{p3}
We have
\begin{equation*}
  \sum_{|l|\leq \log y/2} \mathbb{E}  \bigg|\sum_{n\leq x}f(n)\bigg|^2\exp\bigg(2(k-1)\Re\sum_{p\leq y}\frac{f(p)}{p^{1/2+il/\log y}} +\frac{f(p)^2}{2p^{1+2il/\log y}} \bigg)\geq e^{O(k^2\log \log k)} x(\log y)^{k^2-1}
\end{equation*}
\end{proposition}
\begin{proposition}
\label{p4}
Let us denote the `error'
\begin{equation*}
    \text{Err}_{m,l}(f):=\exp\Big(2(k-1)\Re D_{m,l}(f)\Big)-R_{m,l}(f)=\sum_{\substack{j_1,j_2\geq 0\\ \max(j_1,j_2)>J_m }} \frac{(k-1)^{j_1+j_2}}{j_1!j_2!}(\Re D_{m,l}(f))^{j_1+j_2}.
\end{equation*}
For any $1\leq m\leq M$, we have
\begin{equation*}
    \sum_{|l|\leq \log y/2} \mathbb{E}  \bigg|\sum_{n\leq x}f(n)\bigg|^2\exp\Big(2(k-1)\sum_{\substack{m'=1\\ m'\neq m} }^M\Re D_{m',l}(f)\Big) |\text{Err}_{m,l}(f)|\leq e^{O(k^4)}e^{-J_m} \frac{\log x}{\log y}x(\log y)^{k^2-1} 
\end{equation*}
\end{proposition}
\begin{remark}
    The exact shape of the error terms $e^{O(k^2\log \log )}$ and $e^{O(k^4)}$ are actually unimportant, in both propositions any bounded function in terms of $k$ would suffice. Since it follows easily from the calculation we decided to keep track of the dependence on $k$ (the error terms are not optimal). 

    The important thing is that the lower bound in Proposition 4.1. is bigger than the upper bound in Proposition 4.2. by a large factor when $m=M$. Recalling that $\frac{\log x}{\log y}=C_0$ and $J_M=\frac{C_0}{10^5 k}$, these bounds differ by a factor of $e^{O(k^4)} e^{-C_0/(10^5k)} C_0$, which can be made small for sufficiently large $C_0$.
\end{remark}
\newpage
\begin{proof}[Proof that Proposition \ref{p3} and \ref{p4} imply Proposition 3.1.]
% For simplicity, let us denote
% \begin{equation*}
%     D_{m,l}(f)=\sum_{y_{m-1}< p\leq y_m}\frac{f(p)}{p^{1/2+il/\log y}} +\frac{f(p)^2}{2p^{1+il/\log y}}. 
% \end{equation*}
By definition, we have the following chain of equalities.
\begin{equation*}
\begin{split}
     &\mathbb{E}\bigg|\sum_{n\leq x}f(n)\bigg|^2 R(f) \\
     =&\sum_{|l|\leq \log y/2} \mathbb{E}  \bigg|\sum_{n\leq x}f(n)\bigg|^2\prod_{m=1}^M R_{m,l}(f) \\
    = & \sum_{|l|\leq \log y/2} \mathbb{E}  \bigg|\sum_{n\leq x}f(n)\bigg|^2\prod_{m=1}^M \bigg(\exp\Big(2(k-1)\Re D_{m,l}(f)\Big)-\text{Err}_{m,l}(f)\bigg) \\
    =&   \sum_{|l|\leq \log y/2} \mathbb{E}  \bigg|\sum_{n\leq x}f(n)\bigg|^2\exp\bigg(2(k-1)\sum_{m=1}^M\Re D_{m,l}(f)\bigg)\prod_{m=1}^M\bigg[1- \frac{\text{Err}_{m,l}(f)}{\exp\big(2(k-1)\Re D_{m,l}(f)\big)}\bigg] \\
\end{split}
\end{equation*}
Let $x_m=-\frac{\text{Err}_{m,l} (f)}{\exp\big(2(k-1)\Re D_{m,l}(f)\big)}$. By the definition of $\text{Err}_{m,l}(f)$ and noting that $R_{m,l}(f)\geq 0$, we have $x_m\geq -1$.
We now use a Bernoulli type inequality in the form that if $ x_1,\ldots, x_M\geq -1$ are real numbers, then
\begin{equation}
\label{bernoulli}
    (1+x_1)\cdots (1+x_M)\geq 1-|x_1|-\ldots -|x_M|.
\end{equation}
This can easily be proven by induction on $M$. For $M=1$, \eqref{bernoulli} is trivial. Assume that $(1+x_1)\cdots (1+x_{M-1})\geq 1-|x_1|-\ldots -|x_{M-1}|$. If $1-|x_1|-\cdots -|x_{M-1}|<0$, then \eqref{bernoulli} holds as LHS is non-negative but RHS is negative. Otherwise, using $1+x_M\geq 1-|x_M|$ and the induction hypothesis we obtain
\begin{equation*}
(1+x_1)\cdots (1+x_M)\geq (1-|x_1|-\cdots -|x_{M-1}|)(1+x_M)\geq (1-|x_1|-\cdots -|x_{M-1}|)(1-|x_M|)\geq  1-|x_1|-\ldots -|x_M|.
\end{equation*}
We thus deduce
\begin{equation*}
\begin{split}
    &   \sum_{|l|\leq \log y/2} \mathbb{E}  \bigg|\sum_{n\leq x}f(n)\bigg|^2\exp\bigg(2(k-1)\sum_{m=1}^M\Re D_{m,l}(f)\bigg)\prod_{m=1}^M\bigg[1- \frac{\text{Err}_{m,l}(f)}{\exp\big(2(k-1)\Re D_{m,l}(f)\big)}\bigg] \\
    \geq & \sum_{|l|\leq \log y/2} \mathbb{E}  \bigg|\sum_{n\leq x}f(n)\bigg|^2\exp\bigg(2(k-1)\sum_{m=1}^M\Re D_{m,l}(f)\bigg)\bigg[1-\sum_{m=1}^M \frac{|\text{Err}_{m,l}(f)|}{\exp\big(2(k-1)\Re D_{m,l}(f)\big)}\bigg].
\end{split}
\end{equation*}

Notice that this expression can be lower bounded using Proposition \ref{p3} and \ref{p4}.
Recall that $\frac{\log x}{\log y}=C_0$ and $J_M=\frac{C_0}{10^5k}$, where $C_0$ is a large absolute constant which we will choose in the next step.
We thus have
\begin{equation*}
 \mathbb{E}\bigg|\sum_{n\leq x} f(n)\bigg|^2 R(f)\geq x(\log y)^{k^2-1}\big(e^{-O(k^2\log \log k)}-e^{O(k^4)-C_0/(10^5k)} C_0\big),
\end{equation*}
where we used that $\sum_{m=1}^M e^{-J_m}\leq 2e^{-J_M}$ (and the factor of $2$ can be absorbed by the error term).
Now we can choose $C_0$ to be sufficiently large in terms of $k$, so that the quantity inside the bracket remains positive, thus Proposition \ref{p1} follows.
\end{proof}
\subsection{Proof of Proposition \ref{p3}}
Note that 
\begin{equation*}
    \frac{f(p)}{p^{1/2+il/\log y}} +\frac{f(p)^2}{2p^{1+2il/\log y}}=-\log\Big(1-\frac{f(p)}{p^{1/2+il/\log y}}\Big)+O\big(p^{-3/2}\big).
\end{equation*}
Since $\sum_{p}p^{-3/2}\ll 1$, we have
\begin{equation}
\label{appending}
\begin{split}
    &\sum_{|l|\leq \log y/2}  \mathbb{E} \bigg|\sum_{n\leq x}f(n)\bigg|^2 \exp\bigg(2(k-1) \Re \sum_{p\leq y } \frac{f(p)}{p^{1/2+2il/\log y}} +\frac{f(p)^2}{2p^{1+il/\log y}}\bigg) \\
    % \geq & e^{O(k)}  \sum_{|l|\leq \log y/2}\mathbb{E} \bigg|\sum_{n\leq x}f(n)\bigg|^2  \exp\bigg(2(k-1) \Re \sum_{p<y } \sum_{v=1}^{\infty} \frac{f(p)^v}{vp^{v(1/2+il/\log y)}} \bigg)  \\
     =& e^{O(k)}  \sum_{|l|\leq \log y/2}  \mathbb{E} \bigg|\sum_{n\leq x}f(n)\bigg|^2 \prod_{p\leq y}\bigg| 1-\frac{f(p)}{p^{1/2+il/\log y}}\bigg|^{-2(k-1)}. \\
    % = & \sum_{|l|\leq \log y/2}  \mathbb{E} \bigg|\sum_{n\leq x}f(n)\bigg|^2  |F_{y}(1/2+il/\log y)|^{2(k-1)}\\
\end{split}
\end{equation}
The Euler product part of the expression behaves nicely, since all the $f(p)$-s are independent of each other by definition. Thus the difficulty is to see how it correlates with the term $\big|\sum_{n\leq x}f(n)\big|^2$.
We use a conditioning argument to proceed. Let $\mathbb{E}^{(y)}$ denote the conditional expectation with respect to $(f(p))_{p\leq y}$, i.e. we condition for primes up to $y$. Note that the Euler product is determined by $(f(p))_{p\leq y}$, so it is considered as a fixed quantity when conditioning. By the tower rule we have $\mathbb{E}=\mathbb{E}\mathbb{E}^{(y)}$, so we need to give a lower bound on $\mathbb{E}^{(y)} \big|\sum_{n\leq x}f(n)\big|^2$. Our next lemma states that this can be lower bounded by roughly speaking the difference of two Euler products.
For simplicity, let us denote
\begin{equation*}
    F_y(s)=\prod_{p\leq y}\Big|1-\frac{f(p)}{p^s}\Big|^{-1}.
\end{equation*}
\newpage
\begin{lemma}
\label{condition_y}
    Assume that $y<x^{1/10}$ is large. Then, for any $\beta>0$, we have
     \begin{equation*}
        \mathbb{E}^{(y)} \bigg|\sum_{n\leq x}f(n)\bigg|^2\gg \frac{x}{\log y} \bigg[  \int_{-1/2}^{1/2} |F_{y}(1/2+\beta+it)|^2 dt-x^{-\beta/4} \int_{-\infty}^{\infty} |F_{y}(1/2+\beta/2+it)|^2 \frac{dt}{1/4+t^2}\bigg],
    \end{equation*}
where the implied constant is absolute.
\end{lemma}
\begin{remark}
    We will choose $\beta=\frac{C}{\log y}$, where $C$ is a sufficiently large constant. With this choice the integrals are roughly the same size, however $x^{-\beta/4}\leq e^{-C/4}$, which becomes small, thus the first term will dominate which gives us a suitable lower bound. 
\end{remark}
\begin{proof}[Proof of Lemma \ref{condition_y}]
The argument follows the proof of Proposition 3 in \cite{harper2020moments}. If $n$ is a positive integer, let $P^{-}(n)$ and $P^{+}(n)$ be its smallest and largest prime factor respectively. 
Any positive integer $n$ can be written uniquely as a product $n=a(n)b(n)$ where $P^{+}(a(n))\leq y$ and $P^{-}(b(n))>y$. Therefore 
\begin{equation*}
    \mathbb{E}^{(y)} \bigg|\sum_{n\leq x}f(n)\bigg|^2=\mathbb{E}^{(y)} \bigg|\sum_{\substack{n\leq x\\ P^{-}(n)>y } }f(n) \sum_{\substack{m\leq x/n \\ P^{+}(m)\leq y }}f(m) \bigg|^2.
\end{equation*}
Letting $c_n= \sum_{\substack{m\leq x/n \\ P^+(m)\leq y } }f(m) $, we see that $c_n$ is determined by $(f(p))_{p\leq y}$, so it may be considered a fixed quantity when taking $\mathbb{E}^{(y)}$. Moreover, if $n_1, n_2$ are such that $P^{-}(n_1n_2)>y$, the orthogonality relation gives $\mathbb{E}^{(y)} f(n_1)\overline{f(n_2)}=\mathbf{1}(n_1=n_2)$. We thus obtain
\begin{equation*}
    \mathbb{E}^{(y)} \bigg|\sum_{\substack{n\leq x\\ P^{-}(n)>y } }f(n) c_n \bigg|^2=\sum_{\substack{n\leq x \\ P^{-}(n)>y }} |c_n|^2.
\end{equation*}
% Applying the orthogonality relation to $\mathbb{E}^{(y)}$, that is $\mathbb{E}^{(y)}f(n)\overbar{f(m)}=f(a(n))f(b(n))\mathbf{1}( b(n)=b(m))$
Since we only we want a lower bound, we may restrict our range of summation to $x^{3/4}/2 <n\leq x$, say, so we get
\begin{equation}
\label{con}
     \mathbb{E}^{(y)} \bigg|\sum_{n\leq x}f(n)\bigg|^2\geq  \sum_{\substack{x^{3/4}/2 <n\leq x\\ P^-(n)>y } }\bigg|  \sum_{\substack{m\leq x/n \\ P^+(m)\leq y }}f(m) \bigg|^2\geq  \sum_{r\leq x^{1/4} }\bigg|  \sum_{\substack{m\leq r \\ P^+(m)\leq y }}f(m) \bigg|^2 \cdot \sum_{\substack{\frac{x}{r+1}<n\leq \frac{x}{r} \\ P^-(n)>y}}  1,
\end{equation}
where the second inequality is obtained via the substitution $r=\big\lfloor \frac{x}{n}\rfloor$. 
\newpage
 The inner sum
\begin{equation*}
     \sum_{\substack{\frac{x}{r+1}<n\leq \frac{x}{r} \\ P^-(n)>y}}  1
\end{equation*}
can be handled by well-known sieve bounds because $y<x^{1/10}$ and the length of the interval $[\frac{x}{r+1}, \frac{x}{r}]$ is $\frac{x}{r}-\frac{x}{r+1}\gg \frac{x}{r^2}\gg x^{1/2}$ when $r\leq x^{1/4}$. To be precise, we apply (the lower bound part of) Theorem 12.5 of \cite{friedlander2010opera} (i.e. the Rosser-Iwaniec sieve) for the sequence $\mathcal{A}=\{n: \frac{x}{r+1}<n\leq \frac{x}{r} \}$ with $z:=y$ and $X=\frac{x}{r}-\frac{x}{r+1}$. The important thing here is that $y<x^{1/10}$ and $r\leq x^{1/4}$ gives us $s:=\frac{\log X}{\log y}>3$, so $f(s)\gg 1$, where $f$ is now the function appearing in the statement of the theorem defined by delayed differential equations (From now on we switch back to $f$ being the random multiplicative function). Also $V(z)=\prod_{p\leq y}\big(1-\frac{1}{p}\big)\gg \frac{1}{\log y}$, thus
\begin{equation*}
     \sum_{\substack{\frac{x}{r+1}<n\leq \frac{x}{r} \\ P^-(n)>y}}  1\gg \frac{x}{r^2\log y}.
\end{equation*}
Hence \eqref{con} is
\begin{equation*}
    \gg 
    \frac{x}{\log y} \sum_{r\leq x^{1/4} } \frac{1}{r^2} \bigg|  \sum_{\substack{m\leq r \\ P^+(m)\leq y }}f(m) \bigg|^2 \gg \frac{x}{\log y} \int_1^{x^{1/4}}\bigg|\sum_{\substack{m\leq t \\ P^+(m)\leq y } }f(m) \bigg|^2 \frac{dt}{t^2}.
\end{equation*}
% By assumption, we have $y<x^{3/4}/2$, therefore
% \begin{equation*}
%     \mathbb{E}^{(y)} \bigg|\sum_{n\leq x}f(n)\bigg|^2=\mathbb{E}^{(y)} \bigg|\sum_{\substack{n\leq x\\ P(n)>y } }f(n) \sum_{\substack{m\leq x/n \\ P(m)\leq y }}f(m) \bigg|^2\geq  \sum_{\substack{y<n\leq x\\ P(n)> x } }\bigg|  \sum_{\substack{m\leq x/n \\ P(m)\leq y }}f(m) \bigg|^2 
%    \geq  \sum_{\substack{x^{3/4}/2 <n\leq x\\ P(n)>y } }\bigg|  \sum_{\substack{m\leq x/n \\ P(m)\leq y }}f(m) \bigg|^2. 
% \end{equation*}
% Since $y\leq x^{1/10}$, standard sieve bounds (e.g. the Rosser-Iwaniec sieve) imply that if $I$ is an interval of length $\gg x^{1/2}$, then
% \begin{equation*}
%     \sum_{\substack{n\in I\\ P^+(n)>y} } 1 \gg \frac{|I|}{\log y}.
% \end{equation*}
% Therefore we can further write
% \begin{equation}
% \label{bum}
%     \sum_{\substack{x^{3/4}/2 <n\leq x\\ P(n)>y } }\bigg|  \sum_{\substack{m\leq x/n \\ P(m)\leq y }}f(m) \bigg|^2 \geq \sum_{r\leq x^{1/4} }\sum_{\substack{\frac{x}{r+1}<n\leq \frac{x}{r} \\ P^+(n)>y}} \bigg|  \sum_{\substack{m\leq r \\ P(m)\leq y }}f(m) \bigg|^2 \gg \frac{x}{\log y} \sum_{r\leq x^{1/4} } \frac{1}{r^2} \bigg|  \sum_{\substack{m\leq r \\ P(m)\leq y }}f(m) \bigg|^2 
% \end{equation}
% For the last inequality, we simply used that
% \begin{equation*}
%     \sum_{\substack{\frac{x}{r+1}<n\leq \frac{x}{r} \\ P^+(n)>y}} 1\gg \frac{x}{r^2 \log y},
% \end{equation*}
% which holds, since $\frac{x}{r}-\frac{x}{r+1}\gg \frac{x}{r^2}\gg x^{1/2}$ when $r\leq x^{1/4}$. 

The integral is in a shape where Lemma \ref{parseval} is applicable. However our integral only goes up to $x^{1/4}$ and not $\infty$, so we introduce a Rankin-type trick to get around. For any $0<\beta<1/10$, we have 
\begin{equation*}
\begin{split}
        \int_1^{x^{1/4}}\bigg|\sum_{\substack{m\leq t \\ P^+(m)<y } }f(m) \bigg|^2 \frac{dt}{t^2} \geq  & \int_1^{\infty} \bigg|\sum_{\substack{m\leq t \\ P^+(m)<y } }f(m) \bigg|^2 \frac{dt}{t^{2+2\beta}}- \int_{x^{1/4}}^{\infty} \bigg|\sum_{\substack{m\leq t \\ P^+(m)<y } }f(m) \bigg|^2 \frac{dt}{t^{2+2\beta} } \\
       \geq & \int_1^{\infty} \bigg|\sum_{\substack{m\leq t \\ P^+(m)<y } }f(m) \bigg|^2 \frac{dt}{t^{2+2\beta}}- x^{-\beta/4} \int_{x^{1/4}}^{\infty} \bigg|\sum_{\substack{m\leq t \\ P^+(m)<y } }f(m) \bigg|^2 \frac{dt}{t^{2+\beta} }\\
     \geq & \frac{1}{2\pi} \bigg(\int_{-1/2}^{1/2} |F_y(1/2+\beta+it)|^2 dt-x^{-\beta/4} \int_{-\infty}^{\infty} |F_y(1/2+\beta/2+it)|^2 \frac{dt}{1/4+t^2} \bigg),\\
\end{split}
\end{equation*}
which proves the lemma (note the restriction $\beta<1/10$ was only needed so that we could drop the denominator in the first integral as it was less than 1). 
\end{proof}
\newpage
 We have thus obtained a lower bound on $\mathbb{E}^{(y)}\big|\sum_{n\leq x}f(n)\big|^2$, so we can continue \eqref{appending} as follows
\begin{equation*}
\begin{split}
     &\mathbb{E} \bigg|\sum_{n\leq x}f(n)\bigg|^2 \sum_{|l|\leq \log y/2} |F_{y}(1/2+il/\log y)|^{2(k-1)} \\
     =&\mathbb{E} \mathbb{E}^{(y)}\bigg|\sum_{n\leq x}f(n)\bigg|^2 \sum_{|l|\leq \log y/2} |F_{y}(1/2+il/\log y)|^{2(k-1)} \\
     \gg & \frac{x}{\log y}\sum_{|l|\leq \log y/2}\bigg[ \int_{-1/2}^{1/2}\mathbb{E}|F_{y}(1/2+\beta+it)|^{2}  |F_{y}(1/2+il/\log y)|^{2(k-1)} dt \\
    & -x^{-\beta/4} \int_{-\infty}^{\infty}\mathbb{E}|F_{y}(1/2+\beta/2+it)|^{2}|F_{y}(1/2+il/\log y)|^{2(k-1)} \frac{dt}{1/4+t^2}\bigg]. \\
\end{split}
\end{equation*}
As discussed in the above remark, we will choose $\beta=\frac{C}{\log y}$, where $C$ is a sufficiently large constant chosen later. This ensures that the factor $x^{-\beta/4}$ is small, so the contribution from the second integral becomes negligible. Note that for the rest of the argument, our implied constants are not allowed to depend on $C$.

We now handle the first integral and show that 
\begin{equation}
\label{positive_expected}
    \int_{-1/2}^{1/2} \mathbb{E}|F_y(1/2+\beta+it)|^{2}  |F_y(1/2+il/\log y)|^{2(k-1)} dt\geq e^{O(k^2\log \log k)} (\log y)^{k^2-1}C^{1-2k}.
\end{equation}
Our expression now involves the correlation of two Euler products, so Lemma 1 is applicable with $\sigma_1=\beta$, $\sigma_2=0$, $t_2=\frac{l}{\log y}.$ and $z=200k^2$. For small primes we have the crude bound
\begin{equation*}
    \prod_{p\leq 200k^2}\Big| 1-\frac{f(p)}{p^{1/2+\beta+it}}\Big|^{-2}\Big| 1-\frac{f(p)}{p^{1/2+\beta+it}}\Big|^{-2(k-1)}=e^{O(k^2)},
\end{equation*}
so Lemma 1 and Mertens' second estimate give us
\begin{equation*}
\begin{split}
    &\mathbb{E}|F_{y}(1/2+\beta+it)|^{2}  |F_{y}(1/2+il/\log y)|^{2(k-1)} \\
    = & \exp\bigg(\sum_{200k^2<p\leq y}\frac{1}{p^{1+2\beta }}+\frac{(k-1)^2}{p}+\frac{2(k-1)\cos\big((t-l/\log y)\log p\big)}{p^{1+\beta}}+O(k^2)\bigg) \\
    = & \exp\bigg(\sum_{p\leq y}\frac{1}{p^{1+2\beta }}+\frac{(k-1)^2}{p}+\frac{2(k-1)\cos\big((t-l/\log y)\log p\big)}{p^{1+\beta}}+O(k^2\log\log k)\bigg). \\
\end{split}
\end{equation*}
To prove \eqref{positive_expected}, we lower bound this expression with an explicit dependence on $C$.
By the inequality $\sum_{p>y} \frac{1}{p^{1+1/\log y} } \ll 1$ and $\log \zeta(1+s)=-\log s +O(1)$ for $s\ll 1$, if $C\geq 1/2$ we have
\begin{equation*}
\begin{split}
    \sum_{p\leq y} \frac{1}{p^{1+2\beta} } \geq &\sum_{p} \frac{1}{p^{1+2\beta} } -\sum_{p\geq y} \frac{1}{p^{1+1/\log y}} \\ 
    =&\log \zeta(1+2\beta) +O(1)  \\
    = & \log \frac{\log y}{2C}+O(1) \\
    = & \log\log y-\log C+O(1). \\
\end{split}
\end{equation*}
Similarly, using $\Re \log \zeta(1+s)=- \log |s|+O(1)$ for $s\ll 1$, when $\big|t-\frac{l}{\log y}\big|\leq \frac{1}{\log y}$ and $C\geq 3$ we obtain
\begin{equation*}
\begin{split}
   \sum_{p\leq y}  \frac{\cos\big((t-l/\log y)\log p\big)}{p^{1+\beta}} &\geq \Re \log \zeta\big(1+\beta+ (t-l/\log y)i\big)+O(1) \\
   &= -\log |\beta +(t-l/\log y)i | +O(1) \\
   &=\log\log y-\frac{1}{2}\log\big(C^2+ (t\log y-l)^2\big)+O(1) \\
   &=\log\log y-\log C+O(1).
\end{split}
\end{equation*}
Therefore, if $t$ satisfies $\big|t-\frac{l}{\log y}\big|\leq \frac{1}{\log y}$, we get
\begin{equation*}
    \mathbb{E}|F_{y}(1/2+\beta+it)|^{2}  |F_{y}(1/2+il/\log y)|^{2(k-1)}\geq e^{O(k^2\log \log k)}(\log y)^{k^2}C^{1-2k},
\end{equation*}
which proves \eqref{positive_expected}.

We now handle the second integral and show
\begin{equation}
\label{secondintegral}
    \int_{-\infty}^{\infty}\mathbb{E}|F_y(1/2+\beta/2+it)|^{2}|F_y(1/2+il/\log y)|^{2(k-1)} \frac{dt}{1/4+t^2}\leq e^{O(k^2\log \log k)}  (\log y)^{k^2-1}.
\end{equation}

For notational simplicity, we assume that $l=0$, the other cases can be handled the same way. For any $t\in \mathbb{R}$ Lemma 1 gives us
\begin{equation*}
    \mathbb{E}|F_y(1/2+\beta/2+it)|^{2}|F_y(1/2)|^{2(k-1)} \leq (\log y)^{(k-1)^2+1}\exp\bigg(2(k-1) \sum_{p\leq y}\frac{\cos(t\log p)}{p^{1+\beta/2}}+O(k^2\log \log k)\bigg). 
\end{equation*}

By Lemma \ref{cosine} we have
\begin{equation*}
\begin{split}
    &\int_{-\infty}^{\infty} \exp\bigg(2(k-1) \sum_{p\leq y}\frac{\cos(t\log p)}{p^{1+\beta/2}}\bigg)\frac{dt}{t^2+1/4} \\
    \leq &e^{O(k)}\bigg[\int_0^{1/\log y} (\log y)^{2(k-1)} dt+\int_{1/\log y}^{10} t^{-2(k-1)} dt+\int_{10}^{\infty} \frac{(\log t)^{2(k-1)} }{t^2}dt\bigg] \leq e^{O(k)}\big( (\log y)^{2k-3}+1\big).
\end{split}
\end{equation*}

Since $k\geq 2$, the first term dominates and thus \eqref{secondintegral} is shown.

% Putting everything together, we see that for any $|l|\leq \log y/2$ we have
% \begin{equation*}
%     x^{-\beta/4} \int_{-\infty}^{\infty}\mathbb{E}|F_y(1/2+\beta/2+it)|^2|F_y(1/2+il/\log y)|^{2\alpha} \frac{dt}{1/4+t^2}\leq e^{O(k^2)} e^{-C/4} (\log y)^{k^2-1}.
% \end{equation*}
 Finally, \eqref{positive_expected} and \eqref{secondintegral} gives us
\begin{equation*}
\begin{split}
    &\int_{-1/2}^{1/2}\mathbb{E}|F_y(1/2+\beta+it)|^2  |F_y(1/2+  il/\log y)|^{2(k-1)}dt - 
    \\ 
    & x^{-\beta/4}  \int_{-\infty}^{\infty}\mathbb{E}|F_y(1/2+\beta/2+it)|^2|F_y(1/2+il/\log y)|^{2(k-1)} \frac{dt}{1/4+t^2} \\
    \geq & (\log y)^{k^2-1}(C^{1-2k}e^{O(k^2\log \log k)}-e^{-C/4} e^{O(k^2\log\log k)} ).\\
\end{split}
\end{equation*}

Now we choose $C$ to be a large multiple of $k^2\log\log k$ to get the proposition.

% For simplicity, let us denote 
% \begin{equation*}
%     F_y(s)=\prod_{p\leq y}\Big|1-\frac{f(p)}{p^s}\Big|^{-1}
% \end{equation*}

\subsection{Proof of Proposition \ref{p4}}

The proof of the proposition will be a relatively straightforward consequence of the following lemma. Fix $1\leq m\leq M$, and let us denote $I_m=(y_{m-1}, y_m]$ and $\mathcal{P}_m=\{p\leq y: \,  p\not\in I_m\}$.
\newpage
\begin{lemma}
\label{hardestlemma}
   For any $t\in \mathbb{R}$ and $j\in \mathbb{N}$, we have
    \begin{equation*}
    \begin{split}
        &\mathbb{E}  \bigg|\sum_{n\leq x}f(n)\bigg|^2 \bigg|\exp\bigg( (k-1)\sum_{p\in \mathcal{P}_m} \frac{f(p)}{p^{1/2+it}}+\frac{f(p)^2}{2p^{1+2it} }\bigg)  \bigg|^2 \bigg| \sum_{p\in I_m} \frac{f(p)}{p^{1/2+it}}+\frac{f(p)^2}{2p^{1+2it} }\bigg|^{2j} \\
        \leq & e^{O(k^4)}x(\log y)^{k^2-2}j! \bigg(\sum_{p\in I_m} \frac{2}{p}+\frac{3}{p^2}\bigg)^j \prod_{p\in I_m}\bigg(1-\frac{1}{p}\bigg)^{-2}\frac{\log x}{\log y}.
    \end{split}
    \end{equation*}
\end{lemma}
\begin{proof}
Our expression involves the product of three terms, where if we ignored the second term (exponential), we could easily apply Lemma \ref{evenmoment}. Therefore we first condition with respect to primes in $\mathcal{P}_m$, which we denote by $\mathbb{E}^{(\mathcal{P}_m)}$, use the tower rule $\mathbb{E}=\mathbb{E}\mathbb{E}^{(\mathcal{P}_m)}$, so our expectation is equal to
    \begin{equation*}
    \begin{split}
        &\mathbb{E}\mathbb{E}^{(\mathcal{P}_m)} \bigg|\sum_{n\leq x}f(n)\bigg|^2 \bigg|\exp\bigg( (k-1)\sum_{p\in \mathcal{P}_m} \frac{f(p)}{p^{1/2+it}}+\frac{f(p)^2}{2p^{1+2it} }\bigg)  \bigg|^2 \bigg| \sum_{p\in I_m} \frac{f(p)}{p^{1/2+it}}+\frac{f(p)^2}{2p^{1+2it} }\bigg|^{2j} \\
        =&\mathbb{E}\bigg|\exp\bigg( (k-1)\sum_{p\in \mathcal{P}_m} \frac{f(p)}{p^{1/2+it}}+\frac{f(p)^2}{2p^{1+2it} }\bigg)  \bigg|^2 \mathbb{E}^{(\mathcal{P}_m)}  \bigg|\sum_{n\leq x}f(n)\bigg|^2 \bigg| \sum_{p\in I_m} \frac{f(p)}{p^{1/2+it}}+\frac{f(p)^2}{2p^{1+2it} }\bigg|^{2j}.
    \end{split}
    \end{equation*}
We would like to apply Lemma \ref{evenmoment} to the inner expectation, for which  we need to separate the terms that involve a prime in $\mathcal{P}_m$. We use the same trick as we did at the beginning of the proof of Lemma \ref{condition_y} and write
\begin{equation*}
\mathbb{E}^{(\mathcal{P}_m)}  \bigg|\sum_{n\leq x}f(n)\bigg|^2 \bigg| \sum_{p\in I_m} \frac{f(p)}{p^{1/2+it}}+\frac{f(p)^2}{2p^{1+2it}} \bigg|^{2j}=\mathbb{E}^{(\mathcal{P}_m)}\bigg|\sum_{\substack{n\leq x \\p|n \implies p\not\in \mathcal{P}_m}}f(n)\sum_{\substack{l\leq x/n \\p|l \implies p\in \mathcal{P}_m}} f(l)\bigg|^2\bigg| \sum_{p\in I_m} \frac{f(p)}{p^{1/2+it}}+\frac{f(p)^2}{2p^{1+2it} }\bigg|^{2j}
\end{equation*}
Let us denote
\begin{equation*}
    c_n=\sum_{\substack{l\leq x/n \\p|l \implies p\in \mathcal{P}_m}} f(l),
\end{equation*}
which is considered a fixed quantity, since we condition on primes in $\mathcal{P}_m$. We are now in the position to apply Lemma \ref{evenmoment}, so letting $\Tilde{d}(n)=\sum_{d|n} \mathbf{1}(p|d \implies p\in I_m)$, we get
\begin{equation*}
    \mathbb{E}^{(\mathcal{P}_m)}\bigg|\sum_{\substack{n\leq x \\p|n \implies p\not\in \mathcal{P}_m}}f(n)c_n\bigg|^2\bigg|\sum_{p\in I_m} \frac{f(p)}{p^{1/2+it}}+\frac{f(p)^2}{2p^{1+2it} }\bigg|^{2j}\ll \bigg(\sum_{\substack{n\leq x \\p|n \implies p\not\in \mathcal{P}_m}}\Tilde{d}(n) |c_n|^2\bigg) j! \bigg(\sum_{p\in I_m}\frac{2}{p}+\frac{3}{p^2}\bigg)^j.
\end{equation*}
Hence we obtained the upper bound
\begin{equation*}
    j! \bigg(\sum_{p\in I_m}\frac{2}{p}+\frac{3}{p^2}\bigg)^j\sum_{\substack{n\leq x \\p|n \implies p\not\in \mathcal{P}_m}}\Tilde{d}(n)\mathbb{E} \bigg|\sum_{\substack{l\leq x/n \\p|l \implies p\in \mathcal{P}_m}} f(l)\bigg|^2 \bigg|\exp\bigg( (k-1)\sum_{p\in \mathcal{P}_m} \frac{f(p)}{p^{1/2+it}}+\frac{f(p)^2}{2p^{1+2it} }\bigg)  \bigg|^2.
\end{equation*}
We first rewrite the exponential as an Euler product and show that 
\begin{equation}
\label{rewrite}
    \bigg|\exp\bigg( (k-1)\sum_{p\in \mathcal{P}_m} \frac{f(p)}{p^{1/2+it}}+\frac{f(p)^2}{2p^{1+2it} }\bigg)  \bigg|^2\leq e^{O(k^2)} \prod_{\substack{p\in \mathcal{P}_m \\ p\geq 10k^2}}\bigg| 1+(k-1)\frac{f(p)}{p^{1/2+it}}+\frac{k(k-1)}{2}\frac{f(p^2)}{p^{1+2it}}\bigg|^2.
\end{equation}
For small primes, we use the crude bound
\begin{equation*}
    \bigg|\exp\bigg( (k-1)\sum_{\substack{p\in \mathcal{P}_m \\ p\leq 10k^2}} \frac{f(p)}{p^{1/2+it}}+\frac{f(p)^2}{2p^{1+2it}}\bigg)  \bigg|^2=e^{O(k^2)}.
\end{equation*}
For any prime $p\geq 10k^2$, the Taylor series expansion of the exponential gives us
    \begin{equation*}
         \exp\bigg( (k-1)\Big(\frac{f(p)}{p^{1/2+it}}+\frac{f(p)^2}{2p^{1+2it}}\Big)\bigg)  = 1+(k-1)\frac{f(p)}{p^{1/2+it}}+\frac{k(k-1)}{2}\frac{f(p^2)}{p^{1+2it}}+O\Big(\frac{k^3}{p^{3/2} }\Big)
    \end{equation*}
    When $p\geq 10k^2$, we have
\begin{equation*}
   \bigg| 1+(k-1)\frac{f(p)}{p^{1/2+it}}+\frac{k(k-1)}{2}\frac{f(p^2)}{p^{1+2it}}\bigg|\geq \frac{1}{2},
\end{equation*}
therefore if $p\geq 10k^2$, then
\begin{equation*}
    \bigg| 1+(k-1)\frac{f(p)}{p^{1/2+it}}+\frac{k(k-1)}{2}\frac{f(p^2)}{p^{1+2it}}+O\Big(\frac{k^3}{p^{3/2} }\Big) \bigg|^2\leq  \bigg| 1+(k-1)\frac{f(p)}{p^{1/2+it}}+\frac{k(k-1)}{2}\frac{f(p^2)}{p^{1+2it}}\bigg|^2\Big( 1+O\Big(\frac{k^3}{p^{3/2}}\Big)\Big),
\end{equation*}
so \eqref{rewrite} follows.

By the orthogonality relation $\mathbb{E} f(n)\overline{f(m)}=\mathbf{1}(n=m)$, for any set of complex coefficients $(a_n)_{n\leq N}$, we have $\mathbb{E}\big| \sum_{n\leq N} a_nf(n)\big|^2=\sum_{n\leq N} |a_n|^2$. Hence, the expectation
\begin{equation*}
    \mathbb{E}\bigg|\sum_{\substack{l\leq x/n \\p|l \implies p\in \mathcal{P}_m}} f(l)\bigg|^2\prod_{\substack{p\in \mathcal{P}_m \\ p\geq 10k^2}}\bigg| 1+(k-1)\frac{f(p)}{p^{1/2+it}}+\frac{k(k-1)}{2}\frac{f(p^2)}{p^{1+2it}}\bigg|^2
\end{equation*}
is equal to
\begin{equation}
\label{square}
    \sum_{u}\bigg|\sum_{\substack{u=lv\\ l\leq x/n\\ p|l \implies p\in \mathcal{P}_m }}\frac{h(v)}{v^{1/2+it} }\bigg|^2,
\end{equation}
where $h$ is a multiplicative function defined as follows. If $p\in \mathcal{P}_m$ and $p\geq 10k^2$, then $h(p)=k-1$ and $h(p^2)=\frac{k(k-1)}{2}$, for any other prime power $h$ takes $0$. Note that if $v$ has a prime factor outside of $\mathcal{P}_m$, then $h(v)=0$. As $h(v)$ is non-negative, by the triangle inequality we may assume $t=0$. Dropping the condition $p|l \implies p\in \mathcal{P}_m, p\geq 10k^2$ and expanding the square, our expression is at most
\begin{equation*}
    \sum_{\substack{l_1v_1=l_2v_2 \\ l_1,l_2\leq x/n }}\frac{h(v_1)h(v_2)}{(v_1v_2)^{1/2} }=\sum_{v_1,v_2}\frac{h(v_1)h(v_2)}{(v_1v_2)^{1/2} }\sum_{\substack{l_1,l_2\leq x/n \\ \frac{l_1}{l_2}=\frac{v_2}{v_1} } }1
\end{equation*}
Let $d=(v_1,v_2)$, $v_1=dw_1$, $v_2=dw_2$. Since $\frac{v_2}{v_1}=\frac{w_2}{w_1}$, where the latter is a reduced fraction, we have
\begin{equation*}
    \sum_{\substack{l_1,l_2\leq x/n \\ \frac{l_1}{l_2}=\frac{v_2}{v_1} } }1=\Big\lfloor \frac{x}{n\max(w_1,w_2)}\Big\rfloor.
\end{equation*}
Noting that $h(ab)\leq h(a)h(b)$ and by symmetry we may assume that $w_1\geq w_2$ at the cost of a factor of 2, our expression is at most 
\begin{equation*}
    2\sum_{\substack{d,w_1,w_2 \\ w_1\geq w_2}}\frac{h(d)^2h(w_1)h(w_2)}{d(w_1w_2)^{1/2} }\frac{x}{n\max(w_1,w_2)}=\frac{2x}{n}\sum_{d,w_1}\frac{h(d)^2h(w_1)}{dw_1^{3/2} }\sum_{w_2\leq w_1}\frac{h(w_2)}{w_2^{1/2}}.
\end{equation*}
We apply Theorem 2.14 and Corollary 2.15 of \cite{montgomery2007multiplicative}, with  $A:=k^2$, say, we obtain
\begin{equation*}
\begin{split}
    \sum_{w_2\leq w_1} h(w_2)\ll & k^2\frac{w_1}{\log w_1} \prod_{\substack{ p\leq w_1 \\p\in \mathcal{P}_m} }\bigg(1+\frac{k-1}{p}+\frac{k(k-1)}{2p^2}\bigg) \\
    \leq & k^2\frac{w_1}{\log w_1} \prod_{\substack{ p\leq w_1 \\p\leq y }}\bigg(1+\frac{k-1}{p}+\frac{k(k-1)}{2p^2}\bigg) \\
    \leq & e^{O(k^2)}  \frac{w_1}{\log w_1} \big( \log \min(w_1,y)\big)^{k-1} \\
    \leq & e^{O(k^2)}w_1(\log y)^{k-2}. \\
\end{split}
\end{equation*}
By partial summation
\begin{equation*}
    \sum_{w_2\leq w_1} \frac{h(w_2)}{w_2^{1/2}}=\frac{1}{w_1^{1/2}}\sum_{w_2\leq w_1} h(w_2)+\frac{1}{2}\int_1^{w_1} \bigg(\sum_{w_2\leq t}h(w_2)\bigg)\frac{dt}{t^{3/2} }\leq e^{O(k^2)}w_1^{1/2}(\log y)^{k-2}.
\end{equation*}
We have
\begin{equation*}
    \sum_{d}\frac{h(d)^2}{d}\leq \prod_{p\leq y }\bigg(1+\frac{(k-1)^2}{p}+\frac{k^2(k-1)^2}{p^2}\bigg)\leq e^{O(k^4)}(\log y )^{(k-1)^2},
\end{equation*}
and similarly
\begin{equation*}
    \sum_{w_1}\frac{h(w_1)}{w_1}\leq \prod_{p\leq y}\bigg(1+\frac{k-1}{p}+\frac{k(k-1)}{2p^2}\bigg)\leq e^{O(k^2)}(\log y)^{k-1}
\end{equation*}
We have established that \eqref{square} is at most
\begin{equation*}
   \frac{x}{n} e^{O(k^4)}(\log y)^{k^2-2}.
\end{equation*}
Thus
\begin{equation*}
    \sum_{\substack{n\leq x \\p|n \implies p\not\in \mathcal{P}_m}}\Tilde{d}(n) \mathbb{E}\bigg|\sum_{\substack{l\leq x/n \\p|l \implies p\in \mathcal{P}_m}} f(l)\bigg|^2 \bigg|\exp\bigg( (k-1)\sum_{p\in \mathcal{P}_m} \frac{f(p)}{p^{1/2}}+\frac{f(p)^2}{2p}\bigg)  \bigg|^2\leq e^{O(k^4)}x(\log y)^{k^2-2}\sum_{\substack{n\leq x \\p|n \implies p\not\in \mathcal{P}_m}}\frac{\Tilde{d}(n)}{n}.
\end{equation*}
Finally, recalling that $\Tilde{d}(n)=\sum_{d|n} \mathbf{1}(p|d \implies p\in I_m)$ and that $p\leq x, p\not \in \mathcal{P}_m$ implies $p\in I_m$ or $p\in [y,x]$, we get
\begin{equation*}
    \sum_{\substack{n\leq x \\p|n \implies p\not\in \mathcal{P}_m}}\frac{\Tilde{d}(n)}{n}=\sum_{\substack{d:\\ p|d\implies p\in I_m} }\frac{1}{d}\sum_{\substack{n\leq x/d \\ p|n\implies p\not\in \mathcal{P}_m }}\frac{1}{n}\leq\prod_{p\in I_m}\Big(1-\frac{1}{p}\Big)^{-2}\prod_{y\leq p\leq x}\Big(1-\frac{1}{p}\Big)^{-1},
\end{equation*}
from which the lemma follows since $\prod_{y\leq p\leq x}\big(1-\frac{1}{p}\big)^{-1}\ll \frac{\log x}{\log y}$.
\end{proof}
Let us turn to the proof of Proposition \ref{p4}. Fix any $1\leq m \leq M$ and $|l|\leq \frac{\log y}{2}$. Recall that
\begin{equation*}
\text{Err}_{m,l}(f)=\sum_{\substack{j_1,j_2\geq 0\\ \max(j_1,j_2)>J_m }} \frac{(k-1)^{j_1+j_2}}{j_1!j_2!}(\Re D_{m,l}(f))^{j_1+j_2}.
\end{equation*}
For simplicity, let us denote
\begin{equation*}
    K_{m,l}(f)= \bigg|\sum_{n\leq x}f(n)\bigg|^2\exp\Big(2(k-1)\sum_{\substack{m'=1\\ m'\neq m} }^M\Re D_{m',l}(f)\Big)=\bigg|\sum_{n\leq x}f(n)\bigg|^2 \bigg|\exp\bigg( (k-1)\sum_{p\in \mathcal{P}_m} \frac{f(p)}{p^{1/2+il/\log y}}+\frac{f(p)^2}{2p^{1+2il/\log y} }\bigg)  \bigg|^2.
\end{equation*}
\newpage
In Lemma \ref{hardestlemma} we have shown that for any integer $j\geq 0$ we have
\begin{equation*}
     \mathbb{E}  K_{m,l}(f)|\Re D_{m,l}(f)|^{2j}\leq  e^{O(k^4)}x(\log y)^{k^2-2}j! A_m^{j+1/2} \prod_{p\in I_m}\bigg(1-\frac{1}{p}\bigg)^{-2}\frac{\log x}{\log y},
\end{equation*}
where we denoted $A_m=\sum_{p\in I_m}\big(\frac{2}{p}+\frac{3}{p^2}\big)$ (one should think of $A_m$ as a small quantity). Unfortunately, $\text{Err}_{m,l}(f)$ involves odd powers of $\Re D_{m,l}(f)$, however applying the Cauchy-Schwarz inequality, for any integer $j\geq 0$ we have
\begin{equation*}
\begin{split}
    \mathbb{E}  K_{m,l}(f)|\Re D_{m,l}(f)|^{2j+1}\leq & ( \mathbb{E}  K_{m,l}(f)|\Re D_{m,l}(f)|^{2j})^{1/2}( \mathbb{E}  K_{m,l}(f)|\Re D_{m,l}(f)|^{2(j+1)})^{1/2}  \\
    \leq &  e^{O(k^4)}x(\log y)^{k^2-2}(j+1)! A_m^{j+1/2} \prod_{p\in I_m}\bigg(1-\frac{1}{p}\bigg)^{-2}\frac{\log x}{\log y},
\end{split}
\end{equation*}
which gives us
\begin{equation*}
\begin{split}
    &\mathbb{E} K_{m,l}(f) |\text{Err}_{m,l}(f)| \\
    \leq & e^{O(k^4)}x(\log y)^{k^2-2} \frac{\log x}{\log y}\prod_{p\in I_m}\bigg(1-\frac{1}{p}\bigg)^{-2} \sum_{\substack{j_1,j_2\geq 0\\ \max(j_1,j_2)>J_m }} \frac{(k-1)^{j_1+j_2}}{j_1!j_2!}\Big\lceil \frac{j_1+j_2}{2} \Big\rceil !  A_m^{\frac{j_1+j_2}{2}}. \\
\end{split}
\end{equation*}
Note that if $m=1$, then $\prod_{p\in I_m}\big(1-\frac{1}{p}\big)^{-2}\ll (\log \log y)^2\leq e^{J_1}$, if $m\geq 2$, then $\prod_{p\in I_m}\big(1-\frac{1}{p}\big)^{-2}\leq 100\leq e^{J_M}\leq e^{J_m}$ so it is enough to show that the inner sum is at most $e^{-2J_m}$ (note that we have a lot of room to spare with these estimates).
By symmetry, we may assume that $j_1\geq j_2$ at the cost of a factor of 2. This implies that $\big\lceil \frac{j_1+j_2}{2} \big\rceil ! \leq j_1^{\frac{j_1+j_2}{2}}$, so we get 
\begin{equation*}
    \sum_{\substack{j_1,j_2\\ \max(j_1,j_2)>J_m }} \frac{(k-1)^{j_1+j_2}}{j_1!j_2!}\Big\lceil \frac{j_1+j_2}{2} \Big\rceil ! A_m^{\frac{j_1+j_2}{2}} \leq 2\sum_{j_1>J_m} \frac{(k-1)^{j_1} }{j_1!} (A_mj_1)^{j_1/2}\sum_{j_2\leq j_1}^{\infty} \frac{(k-1)^{j_2}}{j_2!} (A_mj_1)^{j_2/2}
\end{equation*}
Extending the inner sum to $\sum_{j_2=0}^{\infty}$, this is at most
\begin{equation*}
    2\sum_{j_1>J_m} \frac{(k-1)^{j_1} }{j_1!} (A_mj_1)^{j_1/2} \exp\big( (k-1)(A_mj_1)^{1/2}\big).
\end{equation*}
We claim that
\begin{equation*}
    10^4(k-1)^2A_m\leq J_m.
\end{equation*}
Indeed, if $m=1$, then $A_m\asymp \log \log y$ and $J_1=(\log \log y)^{3/2}$. For $m\geq 2$ we note that $A_m\leq 10$ and $J_m\geq J_M\geq 10^{10}k^2$.

Therefore $\exp\big( (k-1)(A_mj_1)^{1/2}\big)\leq e^{j_1}$. Using $\frac{1}{j_1!}\leq \big( \frac{e}{j_1}\big)^{j_1}$, we arrive at the upper bound
\begin{equation*}
    2\sum_{j_1>J_m} \Big(\frac{e^2(k-1)A_m^{1/2}}{ j_1^{1/2}}\Big)^{j_1}.
\end{equation*}

Again, using $10^4(k-1)^2A_m\leq J_m<j_1$, this is at most
\begin{equation*}
    2\sum_{j_1>J_m} 10^{-j_1}\leq e^{-2J_m},
\end{equation*}

which is what we wanted to show.
% \begin{equation*}
% \begin{split}
%     & \sum_{\substack{j_1,j_2\\ \max(j_1,j_2)>J_m }} \frac{(k-1)^{j_1+j_2}}{j_1!j_2!}\Big\lceil \frac{j_1+j_2}{2} \Big\rceil ! A_m^{\frac{j_1+j_2}{2}} \\
%     \leq & 2\sum_{j_1>J_m} \frac{(k-1)^{j_1} }{j_1!} (A_mj_1)^{j_1/2}\sum_{j_2=0}^{\infty} \frac{(k-1)^{j_2}}{j_2!} (A_mj_1)^{j_2/2} \\
%     =& 2\sum_{j_1>J_m} \frac{(k-1)^{j_1} }{j_1!} (A_mj_1)^{j_1/2} \exp\big( (k-1)(A_mj_1)^{1/2}\big) \\
%     \leq & 2\sum_{j_1>J_m} \Big(\frac{e(k-1)e^{k-1}A_m^{1/2}}{ j_1^{1/2}}\Big)^{j_1} \\
%     \leq & e^{-2J_m}. \\
% \end{split}
% \end{equation*}

\section{Upper bound}
% \begin{itemize}
%     \item Step 1: Reduce the problem to fixed $|l_1|,|l_2|\leq \log y/2$ by a simple application of Hölder
%     \item Step 2: Partition the characters into classes according to the sizes of $\Re D_{m,l_2}(\chi)$, where $m$ runs between 1 and $M$.
%     \item Step 3: Replace the the truncated 'real power' exponential quantity by the appropriate Dirichlet polynomial which has bigger size
%     \item Step 4: Replace the expression by the corresponding quantity involving the random multiplicative function
%     \item Step 5: By the "usual method" (i.e. extending to exponential) show that this expression is not too large.
% \end{itemize}
This section is devoted to the proof of Proposition \ref{p2}, in fact we will prove the slightly stronger bound
    \begin{equation*}
        \frac{1}{\phi(q)}\sum_{\chi\,\text{mod}\, q} R(\chi)^{\frac{k}{k-1}}\ll_k (\log y)^{k^2+1}.
    \end{equation*}

Recall that $R(\chi)$ is the sum of (short) Dirichlet polynomials and we raise it to the non-integer $\frac{k}{k-1}$-th power, so we cannot immediately replace this by the corresponding random multiplicative function quantity. However, we are able to replace the fractional power by a genuine Dirichlet polynomial whose size is (up to a small multiplicative error) larger, and then use the random multiplicative function correspondence.  

Before we do that, we first apply the inequality
\begin{equation*}
    (x_1+\cdots + x_M)^{\frac{1}{k-1}}\leq x_1^{\frac{1}{k-1}}+\cdots+ x_M^{\frac{1}{k-1}},\;\;\; \text{ where } k\geq 2 \text{ and } x_1,\ldots, x_M\geq 0.
\end{equation*}
\newpage
 Note that it is crucial that $k\geq 2$, so $\frac{1}{k-1}\leq 1$. Note that $\frac{k}{k-1}=1+\frac{1}{k-1}$, so this inequality implies
\begin{equation}
\label{höldi}
\begin{split}
      R(\chi)^{\frac{k}{k-1}}=\bigg( \sum_{|l|\leq (\log y)/2} \prod_{m=1}^M R_{m,l}(\chi)\bigg)^{\frac{k}{k-1}}\leq & \bigg( \sum_{|l_1|\leq (\log y)/2}\prod_{m=1}^M R_{m,l_1}(\chi) \bigg)\bigg( \sum_{|l_2|\leq (\log y)/2}\prod_{m=1}^M R_{m,l_2}(\chi) \bigg)^{\frac{1}{k-1}} \\
      \leq & \sum_{|l_1|,|l_2|\leq (\log y) /2} \prod_{m=1}^M R_{m,l_1}(\chi)R_{m,l_2}(\chi)^{\frac{1}{k-1}}. \\
\end{split}
\end{equation}

Note that this step of the argument is inspired by the proof of Key Proposition 1 in \cite{harper2020moments}.
From now on (until the very end), we fix the values of $l_1$ and $l_2$ and aim to arrive at the bound
\begin{equation*}
    \frac{1}{\phi(q)}\sum_{\chi\,\text{mod}\, q}\prod_{m=1}^M R_{m,l_1}(\chi)R_{m,l_2}(\chi)^{\frac{1}{k-1}}\ll \frac{(\log y)^{k^2} }{|l_1-l_2|^{2(k-1)} +1}.
\end{equation*}

Summing over $l_1$ and $l_2$, noting that $k-1\geq 1$, we immediately get the proposition. Note that we get a stronger upper bound as $|l_1-l_2|$ gets larger, and this is the part of the argument where it becomes important that $R(\chi)$ is the sum of a Dirichlet polynomial at the shifts $\{1/2+il/\log y: |l|\leq \log y/2 \}$ and not just at one point.

We now want to replace $R_{m,l_2}(\chi)^{\frac{1}{k-1}}$ by a Dirichlet polynomial that has roughly the same size and also short enough that we can switch to the corresponding random multiplicative quantities. It turns out that the shape of this polynomial depends on the size of $\Re D_{m,l_2}(\chi)$, therefore we partition the set of characters into subclasses according to these sizes.

For any $1\leq m\leq M$, we partition $[0,\infty)$ into intervals $(I^{(m)}_n)_{n\geq 0}$, where $I^{(m)}_0=[0, \frac{J_m}{100k}] $, and for any $n\geq 1$ we define the dyadic interval $I^{(m)}_n=\frac{J_m}{100k}\cdot [2^{n-1}, 2^n]$. We let
\begin{equation*}
    \mathcal{X}(n_1,\ldots, n_M)=\{ \chi\in \mathcal{X}_q : \forall 1\leq m\leq M \,\,  |\Re D_{m,l_2}(\chi)|\in I^{(m)}_{n_m}\}.
\end{equation*}

Since the `expected value' of $|\Re D_{m,l_2}(\chi)|$ is much less than $\frac{J_m}{100k}$, most characters lie in $\mathcal{X}(0,\ldots, 0)$ and we expect to get the main contribution to our sum from here.

Let us now fix non-negative integers $n_1,\ldots n_M$ and consider $\chi\in \mathcal{X}(n_1,\ldots, n_M)$. For simplicity let us denote $W_m=\inf I^{(m)}_{n_m}$.
Our quantity that `replaces' $R_{m,l_2}(\chi)^{\frac{1}{k-1}}$ is defined as follows. Let $a_m= 2\lceil 200kJ_m \rceil$.
\begin{equation*}
     U_{m,l_2}(\chi)=
         \begin{cases}
             \Big( \sum_{j=0}^{J_m}\frac{1}{j!}\big( \Re D_{m,l_2}(\chi)\big)^j\Big)^2 & \text{ if }  n_m=0, \\
             e^{4 W_m}|D_{m,l_2}(\chi)W_m^{-1}|^{a_m}  & \text{ if } J_m/100k\leq W_m \leq 100kJ_m, \\
             \Big(2\frac{ (k-1)^{J_m} }{J_m!} (2W_m)^{J_m}\Big)^{\frac{2}{k-1} } |D_{m,l_2}(\chi) W_m^{-1}|^{a_m} & \text{ if } 100kJ_m\leq W_m.\\
         \end{cases}
\end{equation*}
\begin{remark}
    Note that $U_{m,l_2}(\chi)$ is non-negative. When $n_m=0$, i.e. $|\Re D_{m,l_2}(\chi)|$ is its `typical' size, both $R_{m,l_2}(\chi)^{\frac{1}{k-1}}$ and $U_{m,l_2}(\chi)$ are truncated versions of the same exponential, and since we take enough terms in the truncation they are close.
    
    When $n_m\geq 1$, $|D_{m,l_2}(\chi) W_m^{-1}|^{a_m}$ is a penalising term inserted to detect the fact that $n_m\geq 1$ is rare. It turns out that $W_m^{-a_m}$ will outweigh all the other terms in our calculation so an appropriate saving will be obtained.
\end{remark}

In the next lemma we show that $U_{m,l_2}(\chi)$ is larger than $R_{m,l_2}(\chi)^{\frac{1}{k-1}}$ up to a small multiplicative error.
\begin{lemma}
\label{easy}
For any $\chi$ mod $q$ and $1\leq m\leq M$, we have
\begin{equation*}
       R_{m,l_2}(\chi)^{\frac{1}{k-1}} \leq \big(1+O(e^{-J_m})\big)U_{m,l_2}(\chi).
\end{equation*}
\end{lemma}
\begin{proof}
First, assume that $n_m=0$, i,e, $|\Re D_{m,l_2}(\chi)|\leq \frac{J_m}{100k}$.  We then have
\begin{equation*}
    R_{m,l_2}(\chi)=\Big(\sum_{j=0}^{J_m}\frac{(k-1)^j}{j!}(\Re D_{m,l_2}(\chi))^j\Big)^2=e^{2(k-1)\Re D_{m,l_2}(\chi)}\big(1+O(e^{-J_m})\big).
\end{equation*}
\newpage
On the other hand
\begin{equation*}
    U_{m,l_2}(\chi)=\Big( \sum_{j=0}^{J_m}\frac{1}{j!}\big( \Re D_{m,l_2}(\chi)\big)^j\Big)^2=e^{2\Re D_{m,l_2}(\chi)}\big(1+O(e^{-J_m})\big),
\end{equation*}
from which the lemma follows. 

\noindent Let us now handle the case $n_m\geq 1$. By definition $W_m\leq |\Re D_{m,l_2}(\chi)|\leq 2W_m$, so
\begin{equation*}
    R_{m,l_2}(\chi)^{\frac{1}{k-1}}\leq \Big(\sum_{j=0}^{\infty}\frac{(k-1)^j}{j!} (2W_m)^j\Big)^{\frac{2}{k-1}}= e^{4W_m},
\end{equation*}
and also $|\Re D_{m,l_2}(\chi) W_m^{-1}|\geq 1$. So if $\frac{J_m}{100k}\leq W_m\leq 100kJ_m$ we are done.

\noindent When $ W_m \geq 100kJ_m$, we start with the trivial bound
\begin{equation*}
     R_{m,l_2}(\chi)\leq \Big(\sum_{j=0}^{J_m}\frac{(k-1)^j}{j!}(2W_m)^j\Big)^2,
\end{equation*}
moreover each term in the sum is at least twice as big as the previous one, hence 
\begin{equation*}
    \sum_{j=0}^{J_m}\frac{(k-1)^j}{j!}(2W_m)^j\leq 2\frac{(k-1)^{J_m} }{J_m!}(2W_m)^{J_m},
\end{equation*}
from which the lemma follows.
\end{proof}
\begin{proposition}
\label{big_upper}
    Fix non-negative integers $n_1,\ldots n_M$. We then have
\begin{equation*}
    \frac{1}{\phi(q)}\sum_{\chi\in \mathcal{X}(n_1,\ldots,n_M)} \prod_{m=1}^M R_{m,l_1}(\chi)U_{m,l_2}(\chi)\ll \frac{(\log y)^{k^2}}{|l_1-l_2|^{2(k-1)}+1} \prod_{m=1}^M (\inf I^{(m)}_{n_m}+1)^{-2}.
\end{equation*}
\end{proposition}
\begin{proof}[Proof that Proposition \ref{big_upper} implies Proposition \ref{p2}]
   Firstly, note that $\prod_{m=1}^M (1+O(e^{-J_m}))\ll 1 $, so by Lemma \ref{easy} we have
 \begin{equation}
\label{fractal}
     \frac{1}{\phi(q)}\sum_{\chi\in \mathcal{X}(n_1,\ldots,n_M)} \prod_{m=1}^M R_{m,l_1}(\chi)R_{m,l_2}(\chi)^{\frac{1}{k-1}}\ll  \frac{1}{\phi(q)} \sum_{\chi\in \mathcal{X}(n_1,\ldots,n_M)} \prod_{m=1}^M  R_{m,l_1}(\chi) U_{m,l_2}(\chi). 
     % \leq &  \mathbb{E} \prod_{m=1}^M R_{m,l_1}(f) U_m(f) \\
     % = &\prod_{m=1}^M \mathbb{E} R_{m,l_1}(f) U_{m,l_2}(f). \\
 \end{equation}
\newpage
\noindent We also have
\begin{equation*}
\begin{split}
    \sum_{n_1,\ldots , n_m\geq 0}\prod_{m=1}^M (\inf I^{(m)}_{n_m}+1)^{-2} =&\prod_{m=1}^M \Big(\sum_{n_m=0}^{\infty}(\inf I^{(m)}_{n_m}+1)^{-2}\Big) \\
    \leq & \prod_{m=1}^M\Big(1+2\cdot \big(\frac{100k}{J_m}\big)^2\Big) \\
    \leq &\exp\Big(10^5k^2 \sum_{m=1}^M J_m^{-2}\Big) \\
    \ll & 1.
\end{split}
\end{equation*}
These two inequalities together with Proposition \ref{big_upper} imply
\begin{equation*}
    \frac{1}{\phi(q)}\sum_{\chi\,\text{mod}\, q} \prod_{m=1}^M R_{m,l_1}(\chi)R_{m,l_2}(\chi)^{\frac{1}{k-1}}\ll \frac{(\log y)^{k^2}}{|l_1-l_2|^{2(k-1)}+1},
\end{equation*}
% \begin{equation*}
% \begin{split}
%     \sum_{n_1,\ldots ,n_M\geq 0}  \exp\big(-\inf I^{(1)}_{n_1}-\ldots -\inf I^{(M)}_{n_M}\big) = & \prod_{m=1}^M \sum_{n\geq 0} \exp\big(-\inf I^{(m)}_{n}\big) \\
%     \leq &\prod_{m=1}^M \big(1+2e^{-\frac{J_m}{100k} }\big) \\
%     \leq &\exp\big( 2\sum_{m=1}^M e^{-\frac{J_m}{100k} } \big) \\
%     \ll  &1. \\
% \end{split}
% \end{equation*}
% Therefore by Proposition \ref{big_upper} and \eqref{fractal}
% \begin{equation*}
%      \frac{1}{\phi(q)}\sum_{\chi\in \mathcal{X}_q} \prod_{m=1}^M R_{m,l_1}(\chi)R_{m,l_2}(\chi)^{\frac{1}{k-1}}\ll  \frac{(\log y)^{k^2}}{|l_1-l_2|^{2(k-1)}+1},
% \end{equation*}
so by \eqref{höldi} we get
\begin{equation*}
\begin{split}
    \frac{1}{\phi(q)}\sum_{\chi\,\text{mod}\, q} R(\chi)^{\frac{k}{k-1}}  \leq &
  \sum_{|l_1|,|l_2|\leq \log y/2} \frac{1}{\phi(q)}\sum_{\chi\,\text{mod}\, q} \prod_{m=1}^M R_{m,l_1}(\chi)R_{m,l_2}(\chi)^{\frac{1}{k-1}} \\ 
   \ll &\sum_{|l_1|,|l_2|\leq \log y/2}  \frac{(\log y)^{k^2}}{|l_1-l_2|^{2(k-1)}+1} \\
   \ll & (\log y)^{k^2+1}, \\
\end{split}
\end{equation*}
so Proposition \ref{p2} follows.
\end{proof}
 To show Proposition \ref{big_upper} we first state and prove two lemmas. Note that we are switching to the random multiplicative function for now. We define the polynomials $R_{m,l}(f)$ and $U_{m,l}(f)$ the same way, except that $\chi$ is replaced by $f$.
 \begin{lemma}
 \label{zerocase}
     For $U_{m,l_2}$ in the case where $n_m=0$, we have 
     \begin{equation*}
          \mathbb{E} \big|e^{2(k-1)\Re D_{m,l_1}(f)+2\Re D_{m,l_2}(f)}- R_{m,l_1}(f) U_{m,l_2}(f)\big|\leq e^{-J_m}.
     \end{equation*}
 \end{lemma}
\newpage
 \begin{remark}
     This is the same proof method as in the previous section where we replace a truncated exponential by the actual exponential at the cost of a small error.
 \end{remark}

\begin{proof}
By definition $R_{m,l_1}(f) U_{m,l_2}(f)=\Big(\sum_{j=0}^{J_m}\frac{(k-1)^j}{j!}(\Re D_{m,l_1}(f))^j\Big)^2\Big(\sum_{j=0}^{J_m} \frac{1}{j!}(\Re D_{m,l_2}(f))^j\Big)^2$. Using the Taylor series of the exponential we may bring $e^{2(k-1)\Re D_{m,l_1}(f)+2\Re D_{m,l_2}(f)}$ into similar form, except that the summation goes to infinity in both brackets. By taking absolute values we thus obtain
\begin{equation}
\label{difference}
    \big| e^{2(k-1)\Re D_{m,l_1}(f)+2\Re D_{m,l_2}(f)}- R_{m,l_1}(f) U_{m,l_2}(f)\big| \leq \sum_{\substack{j_1,j_2,j_3,j_4 \\ \max(j_1,j_2,j_3,j_4)>J_m}}\frac{(k-1)^{j_1+j_2}}{j_1!j_2!j_3!j_4!} |D_{m,l_1}(f)|^{j_1+j_2}|D_{m,l_2}(f)|^{j_3+j_4}.
\end{equation}
We apply Cauchy-Schwarz and Lemma \ref{evenmoment} to get
\begin{equation*}
\begin{split}
    \mathbb{E}|D_{m,l_1}(f)|^{j_1+j_2}|D_{m,l_2}(f)|^{j_3+j_4}\leq & \big(\mathbb{E}|D_{m,l_1}(f)|^{2(j_1+j_2)}\big)^{1/2} \big(\mathbb{E}|D_{m,l_1}(f)|^{2(j_3+j_4)}\big)^{1/2} \\
    \ll & [(j_1+j_2)!(j_3+j_4)!]^{1/2}
     \Big(\sum_{y_{m-1}<p\leq y_m}\frac{2}{p}+\frac{3}{p^2}\Big)^{(j_1+j_2+j_3+j_4)/2}.
\end{split}
\end{equation*}
For simplicity, let us denote $A_m=(k-1)^2 \sum_{y_{m-1}<p\leq y_m} \big(\frac{2}{p}+\frac{3}{p^2} \big)$. We will see that $A_m$ is a fairly small quantity.

Taking expectations in \eqref{difference}, noting that $k-1\geq 1$ and $n!\leq n^n$, we arrive at the upper bound
\begin{equation}
\label{thisone}
    \sum_{\substack{j_1,j_2,j_3,j_4 \\ \max(j_1,j_2,j_3,j_4)>J_m}} \frac{A_m^{(j_1+j_2+j_3+j_4)/2 }}{j_1!j_2!j_3!j_4!} (j_1+j_2)^{(j_1+j_2)/2} (j_3+j_4)^{(j_3+j_4)/2}.
\end{equation}
This is symmetric in the $j_i$-s so we may assume that $\max(j_1,j_2,j_3,j_4)=j_1$ at the cost of a factor of 4. This assumption implies $j_1>J_m$, $j_1+j_2\leq 2j_1$ and $j_3+j_4\leq 2j_1$, so \eqref{thisone} can be upper bounded by
\begin{equation*}
    4\sum_{j_1>J_m} \frac{A_m^{j_1/2} } {j_1!} (2j_1)^{j_1/2}\sum_{j_2,j_3,j_4=0}^{\infty} \frac{(2j_1A_m)^{(j_2+j_3+j_4)/2}  }{j_2!j_3!j_4!}.
\end{equation*}
The inner sum is really the Taylor expansion of the exponential, so we arrive at
\begin{equation*}
\label{thisonetwo}
     4\sum_{j_1>J_m} \frac{(2j_1A_m)^{j_1/2} }{j_1!}e^{3\cdot (2j_1A_m)^{1/2} }
\end{equation*}
We claim that $1000A_m\leq J_m$ (with room to spare). Indeed, if $m=1$, then $A_1\leq 4(k-1)^2 \log\log y$, say, and $J_1=(\log \log y)^{3/2} $. If $m\geq 2$, then (recall $y_{m-1}^{20}=y_m$) $A_m\leq 20(k-1)^2$, so $1000A_m\leq 20000k^2\leq J_M\leq J_m$. Therefore $3\cdot (2j_1A_m)^{1/2}\leq j_1$, so \eqref{thisone} is at most
\begin{equation*}
    4\sum_{j_1>J_m}\Big(\frac{A_m^{1/2} (2j_1)^{1/2} e^2 }{j_1}\Big)^{j_1}
\end{equation*}
Here the numerator $A_m^{1/2} (2j_1)^{1/2} e^2 \leq \frac{j_1}{3}$, which gives the bound
\begin{equation*}
    4\sum_{j_1>J_m}3^{-j_1}\leq e^{-J_m}.
\end{equation*}
% \begin{equation*}
% \begin{split}
%       &\sum_{\substack{j_1,j_2,j_3,j_4 \\ \max(j_1,j_2,j_3,j_4)>J_m}} \frac{A_m^{(j_1+j_2+j_3+j_4)/2 }}{j_1!j_2!j_3!j_4!} (j_1+j_2)^{(j_1+j_2)/2} (j_3+j_4)^{(j_3+j_4)/2} \\
%        & \leq 4\sum_{j_1>J_m} \frac{A_m^{j_1/2} } {j_1!} (2j_1)^{j_1/2}\sum_{j_2,j_3,j_4} \frac{(2j_1A_m)^{(j_2+j_3+j_4)/2}  }{j_2!j_3!j_4!}\\
%        &=  4\sum_{j_1>J_m} \frac{A_m^{j_1/2} }{j_1!}(2j_1)^{j_1/2} e^{3\cdot (2j_1A_m)^{1/2} } \\
%        &\leq 4\sum_{j_1>J_m}\Big(\frac{A_m^{j_1/2} (2j_1)^{1/2} e^2 }{j_1}\Big)^{j_1} \\
%        &\leq e^{-J_m}
% \end{split}
% \end{equation*}
\end{proof}
\begin{lemma}
\label{nonzerocase}
     For $U_{m,l_2}$ in the case where $n_m\geq 1$, we have 
    \begin{equation*}
        \mathbb{E}  R_{m,l_1}(f) U_{m,l_2}(f)\leq (\inf I^{(m)}_{n_m}+1)^{-2}.
    \end{equation*}
\end{lemma}
\begin{remark}
    Here the bound could be made much stronger, however this is enough for us (recall how \ref{big_upper} implies \ref{p2} above). Moreover, in the proof we are allowed to apply crude bounds as well. This comes down to the fact that $n_m\geq 1$ corresponds to rare characters, whose contribution is negligible. 
\end{remark}
\begin{proof}
In this proof, for simplicity let us denote
\begin{equation*}
    P_m=\sum_{y_{m-1}<p\leq y_m}\frac{2}{p}+\frac{3}{p^2}\leq \begin{cases}
        10 \text{  if  } m\geq 2, \\
        3\log \log y \text{  if  } m=1. \\
    \end{cases}
\end{equation*}
By Cauchy Schwarz, we have
\begin{equation*}
    \big(\mathbb{E}  R_{m,l_1}(f) U_{m,l_2}(f) \big)^2\leq \mathbb{E} R_{m,l_1}(f)^2 \mathbb{E}  U_{m,l_2}(f)^2.
\end{equation*}
\newpage
We first bound $\mathbb{E} R_{m,l_1}(f)^2$. When $m\geq 2$ this expectation turns out to be less than a constant depending on $k$. When $m=1$ it will be at most a fixed $k$-dependent power of $\log y$. Taking absolute values we get
\begin{equation*}
    \mathbb{E} R_{m,l_1}(f)^2= \mathbb{E} \bigg( \sum_{j=0}^{J_m}\frac{(k-1)^j}{j!} \big(\Re D_{m,l_1} (f)\big)^j\bigg)^4\leq \mathbb{E} \bigg( \sum_{j=0}^{\infty}\frac{(k-1)^j}{j!}  |D_{m,l_1} (f)|^j\bigg)^4 \\
    =\mathbb{E} e^{4(k-1)|D_{m,l_1} (f)|}.
\end{equation*}   
Note that for any real $x$, we have $e^x\ll \cosh(x)=\sum_{j=0}^{\infty} \frac{x^{2j}}{(2j)!}$. We apply this with $x=4(k-1)|D_{m,l_1} (f)|$ together with Lemma \ref{evenmoment} to obtain
\begin{equation}
\label{kurvaanyad}
\begin{split}
    \mathbb{E} e^{4(k-1)|D_{m,l_1} (f)|}\ll \sum_{j=0}^{\infty} \frac{(4k)^{2j}}{(2j)!}\mathbb{E}|D_{m,l_1} (f)|^{2j}  \ll &\sum_{j=0}^{\infty} \frac{(4k)^{2j}}{(2j)!} j! \Big(\sum_{y_{m-1}<p\leq y_m}\frac{2}{p}+\frac{3}{p^2}\Big)^j \\
    \leq & \sum_{j=0}^{\infty} \frac{(4k)^{2j}}{j!} P_m^j \\
    = & \exp\big( 4k^2P_m\big). \\
\end{split}
\end{equation}
For the second line we used the inequality $(j!)^2\leq (2j)!$.

We next bound $\mathbb{E}  U_{m,l_2}(f)^2$, which will give us the saving for Lemma \ref{nonzerocase}. For simplicity, let us denote $W=\inf I^{(m)}_{n_m}$ (note that $m$ and $n_m$ are fixed in this lemma). 

% We will show that
% \begin{equation*}
%     \mathbb{E}  U_{m,l_2}(f)^2\leq (W+1)^{-10},
% \end{equation*}
% which together with \eqref{kurvaanyad} implies the lemma, since whatever the value of $m$ and $n_m$ are, we have $W\geq \frac{J_1}{100k}\geq \exp(40k^2)$.

$U_{m,l_2}(f)$ has a slightly different shape according to whether $J_m/100k \leq W\leq 100kJ_m$ or $W\geq 100kJ_m$. We first handle $J_m/100k \leq W\leq 100kJ_m$. Here by definition and Lemma \ref{evenmoment} we have
\begin{equation*}
    \mathbb{E}  U_{m,l_2}(f)^2\leq e^{8W}W^{-2a_m} \mathbb{E} |D_{m,l_2}(f)|^{2a_m}\ll e^{8W}W^{-2a_m} a_m! P_m^{a_m}\leq \Big(\frac{e^2a_m P_m}{W^2}\Big)^{a_m}
\end{equation*}
For the last inequality we used $a_m!\leq a_m^{a_m}$ and that $8W\leq 800kJ_m\leq 2a_m$. We claim $\frac{e^2a_m P_m}{W^2}\leq e^{-1}$. Recall that $W\geq J_m/100k$, whereas $a_m\leq 500kJ_m$, say. If $m\geq 2$, then we have $P_m\leq 10$, so
\begin{equation*}
    W^2 a_m^{-1}P_m^{-1}e^{-3}\geq \frac{J_m}{10^{10}k^2}\geq 1.
\end{equation*}
If $m=1$, then we proceed similarly and use $a_1 \leq 500kJ_1$, $W\geq \frac{J_1}{100k}$, $P_1\leq 3\log \log y$, $J_1=(\log \log y)^{3/2}$, hence
\begin{equation*}
    \frac{e^2a_m P_m}{W^2}\leq 10^{10}k^3(\log \log y)^{-1/2},
\end{equation*}
say, which is arbitrarily small if $y$ is large enough. Therefore in the case when $J_m/100k \leq W\leq 100kJ_m$, we have shown that $\mathbb{E}U_{m,l_2}(f)^2\ll e^{-a_m}\leq e^{-W}$, so to prove the lemma we need to show $e^{4k^2P_m}e^{-W}\leq (W+1)^{-4}$, which holds with a lot of room to spare. If $m\geq 2$, we have $P_m\leq 10$ and $W\geq \frac{J_m}{100k}\geq e^{1000k^2}$. When $m=1$ we have $P_1\leq 3\log\log y $ and $W\geq \frac{(\log \log y)^{3/2}}{100k}$.

Now consider the case of $W\geq 100kJ_m$. Using $J_m!\geq \big(\frac{J_m}{e}\big)^{J_m}$ and $\frac{2J_m}{k-1}\leq  200kJ_m\leq a_m/2 $, we have 
\begin{equation*}
    \Big(2\frac{ (k-1)^{J_m} }{J_m!} (2W_m )^{J_m}\Big)^{\frac{2}{k-1} }\leq \Big( \frac{10kW}{J_m}\Big)^{\frac{2J_m}{k-1}}\leq\Big( \frac{10kW}{J_m} \Big)^{a_m/2}.
\end{equation*}
By Lemma \ref{evenmoment} we obtain
\begin{equation*}
    \mathbb{E}U_{m,l_2}(f)^2\leq \Big( \frac{10kW}{J_m} \Big)^{a_m}W^{-2a_m}\mathbb{E} |D_{m,l_2}(f)|^{2a_m}\ll \Big( \frac{10kP_ma_m}{WJ_m} \Big)^{a_m}
\end{equation*}
If $m\geq 2$, we have $10kP_ma_m\leq 10^5k^2J_m$, so
\begin{equation*}
     \Big( \frac{10kP_ma_m}{WJ_m} \Big)^{a_m}\leq \Big(\frac{10^5k^2}{W}\Big)^{a_m}\leq W^{-a_m/2}.
\end{equation*}
If $m=1$, then note that $P_1\leq 3\log \log y$ and $W\geq J_1=(\log \log y)^{3/2}$, so
\begin{equation*}
   \Big( \frac{10kP_ma_m}{WJ_m}\Big)^{a_m} \leq \Big(\frac{10^5k^2\log \log y}{W}\Big)^{a_m}\leq W^{-a_m/4} ,
\end{equation*}
when $y$ is large enough.
Therefore, when $100kJ_m\leq W$ we have shown that $\mathbb{E}U_{m,l_2}(f)^2\ll W^{-100kJ_m}$, so we need to show that $e^{40k^2P_m}W^{-100kJ_m}\leq (W+1)^{-4}$, which again holds with a lot of room to spare.
\end{proof}

\begin{proof}[Proof of Proposition \ref{big_upper}]
     By \eqref{shortpolynomial}, $\prod_{m=1}^M R_{m,l_1}(\chi) U_{m,l_2}(\chi)$ is a Dirichlet polynomial of length at most
 \begin{equation*}
     \prod_{m=1}^M y_m^{(4+4)J_m+2a_m}<q,  
 \end{equation*}
therefore by the orthogonality relation and since $ R_{m,l_1}(\chi) U_{m,l_2}(\chi)$ is a non-negative quantity,
\begin{equation*}
    \frac{1}{\phi(q)} \sum_{\chi\in \mathcal{X}(n_1,\ldots,n_M)}\prod_{m=1}^M R_{m,l_1}(\chi) U_{m,l_2}(\chi) \leq\frac{1}{\phi(q)} \sum_{\chi\in \mathcal{X}_q }\prod_{m=1}^M R_{m,l_1}(\chi) U_{m,l_2}(\chi)= \mathbb{E} \prod_{m=1}^M R_{m,l_1}(f) U_{m,l_2}(f).
\end{equation*}
Note that the random quantities $\big(R_{m,l_1}(f) U_{m,l_2}(f)\big)_{1\leq m\leq M}$ are independent of each other, so we can bring the expectation inside. By Lemma \ref{eulerproduct} we also have
\begin{equation}
\label{faszom1}
\begin{split}
     & \mathbb{E} e^{2(k-1)\Re D_{m,l_1}(f)+2\Re D_{m,l_2}(f)} \\= & e^{O\big(k\sum_{y_{m-1}<p\leq y_m } \frac{1}{p^{3/2} } \big) }\mathbb{E}\prod_{y_{m-1}<p\leq y_m} \Big|1-\frac{f(p)}{p^{1/2+il_1/\log y} } \Big|^{-2(k-1)}\cdot \Big| 1-\frac{f(p)}{p^{1/2+il_2/\log y}}\Big|^{-2} \\
     =&  \exp \Big( \sum_{y_{m-1}<p\leq y_m} \frac{(k-1)^2}{p}+\frac{1}{p}+\frac{2(k-1)\cos \big( \frac{l_1-l_2}{\log y} \log p\big)}{p}+O\big( \frac{k^3}{y_{m-1}^{1/2}} +k\sum_{y_{m-1}<p\leq y_m } \frac{1}{p^{3/2} } \big)\Big) \\
\end{split}
\end{equation}
Hence, this expectation is at least $e^{-O(k^3)}$, so by Lemma \ref{zerocase}, if $n_m=0$, then
\begin{equation}
\label{faszom2}
     \mathbb{E} R_{m,l_1}(f)U_{m,l_2}(f) \leq \exp\big(e^{O(k^3)-J_m}\big) \mathbb{E} e^{2(k-1)\Re D_{m,l_1}(f)+2\Re D_{m,l_2}(f)}.
\end{equation}
Now if we sum these error terms for $1\leq m\leq M$, we get the bound
\begin{equation*}
    \sum_{m=1}^M \Big( e^{O(k^3)-J_m}+\frac{k^3}{y_{m-1}^{1/2}} +k\sum_{y_{m-1}<p\leq y_m } \frac{1}{p^{3/2} }\Big)\ll e^{O(k^3) }.
\end{equation*}
When $n_m>0$, we use Lemma \ref{nonzerocase}, that is $\mathbb{E} R_{m,l_1}(f)U_{m,l_2}(f)\leq (\inf I^{(m)}_{n_m}+1)^{-2}$. Note that when $n_m=0$, we can multiply our bound by $1=(\inf I^{(m)}_{n_m}+1)^{-2}$, as it does not change the value.
Therefore \eqref{faszom1}, \eqref{faszom2} and Lemma \ref{nonzerocase} give 
\begin{equation*}
    \prod_{m=1}^M \mathbb{E} R_{m,l_1}(f)U_{m,l_2}(f)
    \leq \exp \Big( \sum_{p\leq y_M} \frac{(k-1)^2}{p}+\frac{1}{p}+\frac{2(k-1)\cos \big( \frac{l_1-l_2}{\log y} \log p\big)}{p}+O\big(k^3 \big)\Big)\prod_{m=1}^M (\inf I^{(m)}_{n_m}+1)^{-2}.
\end{equation*}
We know apply Lemma \ref{cosine}, to get
\begin{equation*}
    \sum_{p\leq y_M} \frac{(k-1)^2}{p}+\frac{1}{p}+\frac{2(k-1)\cos \big( \frac{l_1-l_2}{\log y} \log p\big)}{p}\ll \frac{(\log y)^{k^2}}{|l_1-l_2|^{2(k-1)}+1},
\end{equation*}
from which the proposition follows.
\end{proof}
\section{The large $x$ case}
In this section we explain how to modify our argument to the case when $q^{1/2}\leq x\leq q/2$. Firstly, we show that we can assume that $x\leq \frac{q}{C}$, where $C$ is a large constant chosen later. By Hölder's inequality we have
\begin{equation*}
    \frac{1}{q-2} \sum_{\chi\neq \chi_0} \bigg| \sum_{n\leq x} \chi(n)\bigg|^{2k}\geq \bigg( \frac{1}{q-2} \sum_{\chi\neq \chi_0} \bigg| \sum_{n\leq x} \chi(n)\bigg|^{2}\bigg)^{k}.
\end{equation*}
Note that we divide by $\phi(q)-1=q-2$ because this is the number of non-trivial characters. By the orthogonality relation, we have 
\begin{equation*}
    \frac{1}{q-1} \sum_{\chi\neq \chi_0} \bigg| \sum_{n\leq x} \chi(n)\bigg|^{2}=\lfloor x\rfloor-\frac{\lfloor x\rfloor^2}{q-1},
\end{equation*}
which is at least $\frac{x}{3}$, say, so we get
\begin{equation*}
 \frac{1}{q-2} \sum_{\chi\neq \chi_0} \bigg| \sum_{n\leq x} \chi(n)\bigg|^{2k}\geq x^k 3^{-k},
\end{equation*}
which is the desired bound when $x\geq q/C$.

When $q^{1/2}\leq x\leq q/C$, we choose our proxy object $R(\chi)$ the same way as if we had $x\leq q^{1/2}$, but the parameter $y$ is chosen as $y=\big( \frac{q}{x}\big)^{1/C_0}$ (and the subdivison $1=y_0<y_1<\ldots <y_M=y$ is chosen accordingly), with the same large constant $C_0$. In particular the proxy objects $R(\chi)$ corresponding to $x$ and $\frac{q}{x}$ are the same.
Now it is enough to show that with this choice of $y$ and $R(\chi)$, Proposition \ref{p1} and \ref{p2} hold, since then \eqref{main_Hölder} (Hölder's inequality) implies 
\newpage
\begin{equation*}
    \frac{1}{\phi(q)}\sum_{\chi \in \mathcal{X}_q^*}\bigg|\sum_{n\leq x} \chi(n)\bigg|^{2k}\gg_k x(\log y)^{(k-1)^2},
\end{equation*}
which is sufficient for us as $\log y \asymp \log  \frac{q}{x}$. Proposition \ref{p2} clearly holds, since the argument does not involve $x$ whatsoever, however the same is not true for Proposition \ref{p1}. At the start of the argument, we added the contribution from the trivial character $\Big|\sum_{n\leq x} \chi_0(n)\Big|^2R(\chi_0)$ so we could use orthogonality and switch to the random multiplicative function. There we used that $x\leq q^{1/2}$. However, now we have
\begin{equation*}
    \Big|\sum_{n\leq x} \chi_0(n)\Big|^2R(\chi_0)\ll x^2(\log y)\prod_{m=1}^M (10ky_m)^{2J_m} \ll x^2 \Big(\frac{q}{x}\Big)^{1/10},
\end{equation*}
so the contribution is
\begin{equation*}
    \ll \frac{x^2}{\phi(q)}  \Big(\frac{q}{x}\Big)^{1/10}\leq 2x \Big(\frac{x}{q}\Big)^{1-1/10}\leq 2xC^{-9/10}.
\end{equation*}
We choose $C$ so large in terms of $k$ that this contribution becomes negligible.

The second thing we need to be careful of when reproving Proposition \ref{p2} is that in Lemma \ref{condition_y} we assumed $y<x^{1/10}$, but this holds now `even more', since $y^{10}\leq \frac{q}{x}\leq x$.
\section{Proof of Theorem \ref{t2}}
We show how to modify the proof of Theorem \ref{t1} to obtain Theorem \ref{t2}. 

Firstly, we consider even characters. Let $x=q^{1/2}(\log q)^2$, $y=x^{1/C_0}$, where $C_0$ is a sufficiently large constant. Our proof method is essentially the same. We define our proxy object $R(\chi)$ corresponding to $y$ as in Section 3, and apply Hölder's inequality
\begin{equation*}
    \bigg(\frac{1}{\phi(q)}\sum_{\chi \in X_q^+\backslash \{ \chi_0\} } |\theta(1,\chi)|^{2k} \bigg)^{1/k} \bigg( \frac{1}{\phi(q)}\sum_{\chi \in X_q^+\backslash \{ \chi_0\} } R(\chi)^{\frac{k}{k-1}}\bigg)^{(k-1)/k}\geq \frac{1}{\phi(q)}\sum_{\chi \in X_q^+\backslash \{ \chi_0\} }  |\theta(1,\chi)|^{2} R(\chi).
\end{equation*}
\newpage
By Section 5, we clearly have the same upper bound
\begin{equation*}
    \frac{1}{\phi(q)}\sum_{\chi \in X_q^+\backslash \{ \chi_0\} } R(\chi)^{\frac{k}{k-1}}\ll (\log y)^{k^2+1},
\end{equation*}
on the other hand, getting the lower bound for
\begin{equation*}
    \frac{1}{\phi(q)}\sum_{\chi \in X_q^+\backslash \{ \chi_0\} }  |\theta(1,\chi)|^{2} R(\chi)
\end{equation*}
is trickier. For any character $\chi$, we have the crude bound
\begin{equation*}
   \Big| \sum_{n\geq x}\chi(n)e^{-\frac{\pi n^2}{q} }\Big|\leq qe^{-\frac{x^2}{q}}\leq e^{-(\log q)^3},
\end{equation*}
so it suffices to show
\begin{equation*}
    \frac{1}{\phi(q)}\sum_{\chi \in X_q^+\backslash \{ \chi_0\} }  \Big|\sum_{n\leq x}\chi(n)e^{-\frac{\pi n^2}{q} } \Big|^{2} R(\chi)\gg q^{1/2}(\log y)^{k^2-1}.
\end{equation*}
As we did at the start of Section 4, we add $\chi_0$ into the summation as its contribution is negligible compared to the lower bound we want to show. Next we use the orthogonality relation for even characters (here we use that $q$ is prime), namely
\begin{equation*}
    \frac{2}{\phi(q)}\sum_{\chi\in \chi_q^+} \chi(n)\Bar{\chi}(m)=\mathbf{1}(n\equiv \pm m\, \text{mod }q),
\end{equation*}
and note that the Dirichlet polynomial $\big|\sum_{n\leq x}\chi(n)e^{-\frac{\pi n^2}{q} } \big|^{2} R(\chi)$ has length at most $x^{1+1/10}<q$ (to be precise, here by length we mean the largest $n$ such that a term $\chi(n)$ or $\Bar{\chi}(n)$ occur), so only the diagonal terms survive and we can switch to proving
\begin{equation*}
    \mathbb{E}\Big| \sum_{n\leq x}f(n)e^{-\frac{-\pi n^2}{q}}\Big|^2 R(f)\gg x(\log y)^{k^2-1},
\end{equation*}
for which we would like bounds corresponding to Proposition \ref{p3} and \ref{p4}. The proof of Proposition \ref{p4} is done in the very same manner as in the unweighted case, the only slight difference is that during the proof of Lemma \ref{hardestlemma}, the quantity \eqref{square} is a bit different with a weight $e^{-\frac{\pi l^2}{q}}$ introduced in the inner sum, but since this weight as at most $1$, this does not change rest of the argument.

On the other hand, the proof of Proposition \ref{p3} with weights $e^{-\frac{\pi n^2}{q} }$ is not as straightforward, because these weights are not multiplicative. After using the same conditioning argument (condition with respect to  primes up to $y$), we need to get a lower bound on
\begin{equation*}
    \mathbb{E}\sum_{\substack{n\leq x \\ P^-(n)>y} } \Big| \sum_{\substack{m\leq x/n \\ P^+(m)\leq y}} e^{-\pi \frac{(mn)^2}{q} } f(m)\Big|^2 \prod_{p\leq y}\Big| 1-\frac{f(p)}{p^{1/2+it_0}}\Big|^{-2(k-1)},
\end{equation*}
where $t_0=\frac{il}{\log y}$ for some $|l|\leq \log y/2$. The actual value of $t_0$ is unimportant as the argument goes the same for all values.
Firstly, note that we may extend the summation over $m$ to infinity by the fast decay of the exponential weight, in particular we have
\begin{equation}
\label{nemtommitirjak}
    \sum_{\substack{n\leq x \\ P^-(n)>y} } \Big| \sum_{\substack{m\leq x/n \\ P^+(n)\leq y}} e^{-\pi \frac{(mn)^2}{q} } f(m)\Big|^2=\sum_{\substack{n\leq x \\ P^-(n)>y} } \Big| \sum_{\substack{m=1 \\ P^+(m)\leq y}}^{\infty} e^{-\pi \frac{(mn)^2}{q} } f(m)\Big|^2 + O\big(e^{-(\log q)^2}\big).
\end{equation}
Similarly to the unweighted case, our next step is to replace the sum over $n$ by an integral (in the unweighted case we actually used a substitution $r= \lfloor x/n \rfloor $ and then switched to an integral). Let us denote
\begin{equation*}
    g(t) = \sum_{\substack{m=1 \\ P^+(m)\leq y}}^{\infty} e^{-\pi \frac{(mt)^2}{q} } f(m).
\end{equation*}
We show that $g(t)$ changes relatively slowly for sufficiently large $t$.
\begin{lemma}
\label{smallchange}
    For any $t>0$ and $0<\alpha<t$, we have
    \begin{equation*}
        |g(t+\alpha)-g(t)|\ll \frac{\alpha q^{1/2}}{t^2}
    \end{equation*}.
\end{lemma}
\newpage
\begin{proof}
Using the approximation $e^{u}=1+O(u)$ when $u\leq 0$, we obtain
\begin{equation*}
    \Big| e^{-\pi \frac{(mt)^2}{q} }-e^{-\pi \frac{m^2(t+\alpha)^2}{q} } \Big|=e^{-\pi \frac{(mt)^2}{q} }\Big|1-e^{-\pi \frac{m^2(2\alpha t+\alpha^2)}{q} }\Big|  \ll  \frac{m^2\alpha t}{q}\cdot e^{-\frac{(mt)^2}{q} }
\end{equation*}
Therefore, by the triangle inequality, we obtain
\begin{equation*}
    |g(t+\alpha)-g(t)|\leq \frac{\alpha t}{q} \sum_{m=1}^{\infty} m^2 e^{-\frac{(mt)^2}{q} }.
\end{equation*}
We have
\begin{equation*}
     \sum_{m=1}^{\infty} m^2 e^{-\frac{(mt)^2}{q} }  =\sum_{k=0}^{\infty} \sum_{\frac{kq^{1/2}}{t} <m \leq \frac{(k+1)q^{1/2}}{t}} m^2 e^{-\frac{(mt)^2}{q} }\ll\frac{q^{3/2}}{t^3} \sum_{k=1}^{\infty} (k+1)^2 e^{-k^2}\ll \frac{q^{3/2}}{t^3},
\end{equation*}
which proves the lemma.
\end{proof}
Fix $t$ in the range $x^{3/4}\leq t\leq x$. Let $\delta=\frac{1}{100}$. Note that $|g(t)|\ll \frac{q^{1/2}}{t}$, so if $t\leq n\leq t+x^{\delta}$, then  $|g(n)|^2-|g(t)|^2 \ll \frac{q^{1/2}}{t}|g(n)-g(t)|\ll \frac{qx^{\delta}}{t^3}$ by the previous lemma. Therefore
\begin{equation*}
    \sum_{\substack{t\leq n\leq t+x^{\delta} \\ P^-(n)>y} }  |g(n)|^2 =\sum_{\substack{t\leq n\leq t+x^{\delta} \\ P^-(n)>y} }  |g(t)|^2 + O\Big(\frac{x^{2\delta} q }{t^3}\Big).
\end{equation*}
Recall that $x=q^{1/2}(\log q)^2$, so the error term is $O(x^{-1/10})$, say. A similar argument yields
\begin{equation*}
    |g(t)|^2=x^{-\delta} \int_t^{t+x^{\delta} } |g(u)|^2du +O\big(x^{-1/10}\big).
\end{equation*}
The same sieve argument as in the proof of Lemma \ref{condition_y} implies (assume $y<x^{1/1000}$, say) that
\begin{equation*}
    \sum_{\substack{t\leq n\leq t+x^{\delta} \\ P^-(n)>y} }  1\gg \frac{x^{\delta}}{\log y},
\end{equation*}
\newpage
\noindent so our last three equations put together imply
\begin{equation*}
    \sum_{\substack{t\leq n\leq t+x^{\delta} \\ P^-(n)>y} }  |g(n)|^2\gg \frac{1}{\log y} \int_t^{t+x^{\delta} } |g(u)|^2du +O\big(x^{-1/10}\big).
\end{equation*}
If we sum this over the appropriate $t$-s we get
\begin{equation*}
    \sum_{\substack{x^{3/4}\leq n\leq x \\ P^{-}(n)>y }} |g(n)|^2\gg \frac{1}{\log y}\int_{x^{3/4} }^x |g(u)|^2du +O\big(x^{1 -1/10}\big)=\frac{q^{1/2}}{\log y}\int_{(\log q)^{-2}}^{q^{1/8}(\log q)^{-3/2} } \frac{|g(q^{1/2}/v)|^2}{v^2}dv+O\big(x^{9/10}\big),
\end{equation*}
where we applied the substitution $v=\frac{q^{1/2}}{u}$ (recall that $x=q^{1/2}(\log q)^2$). For simplicity, let us denote $h(v)=g(q^{1/2} /v)$, so
\begin{equation*}
    h(v)=\sum_{\substack{m=1 \\ P^+(m)\leq y}}^{\infty} e^{-\pi \frac{m^2}{v^2} } f(m).
\end{equation*}
As in the proof of Lemma \ref{condition_y}, we apply Rankin's trick, however we cannot apply Lemma \ref{parseval}, because our sum is weighted, so we make use of (the first part of) Theorem 5.4 in \cite{montgomery2007multiplicative}, which is a Parseval-type identity and a generalisation of Lemma \ref{parseval}. Thus, we let $K(s)=\int_0^{\infty}h(v)v^{-s-1}dv$, so for any constant $\beta>0$ we obtain
\begin{equation*}
\begin{split}
    \frac{q^{1/2}}{\log y}\int_{(\log q)^{-2}}^{q^{1/8}(\log q)^{-3/2} } \frac{|h(v)|^2}{v^2}dv\gg &  \frac{q^{1/2}}{\log y} \bigg[  \int_{-1/2}^{1/2} |K(1/2+\beta+it)|^2 dt-q^{-\beta/9} \int_{-\infty}^{\infty} |K(1/2+\beta/2+it)|^2 \frac{dt}{1/4+t^2}\bigg]+ \\
    &+O\Big(\frac{q^{1/2}}{\log y}\int_{0}^{(\log q)^{-2}} \frac{|h(v)|^2}{v^2}dv \Big) .
\end{split}
\end{equation*}
Here the error term appears because to apply the Parseval-type identity we need to integrate from 0. Luckily, the error term is negligible, since in the range $0< v < (\log q)^{-2}$ we have $|h(v)|\ll e^{-1/v^2}$. As for $K(s)$, some calculation yields
\begin{equation*}
    K(s)=\frac{\Gamma(s/2)}{2\pi^{s/2} }\sum_{\substack{m=1 \\ P^+(m)\leq y}}^{\infty} \frac{f(m)}{m^s} = \frac{\Gamma(s/2)}{2\pi^{s/2} } \prod_{p\leq y}\Big( 1-\frac{f(p)}{p^s}\Big)^{-1},
\end{equation*}
so the argument can be finished the same way as in Section 4 (in fact, the presence of the $\Gamma$ factor makes the technicalities easier).

For odd characters, the modification is almost the same. Here \eqref{nemtommitirjak} is modified to
\begin{equation*}
    \sum_{\substack{n\leq x \\ P^-(n)>y} }n^2 \Big| \sum_{\substack{m=1 \\ P^+(m)\leq y}}^{\infty}m e^{-\pi \frac{(mn)^2}{q} } f(m)\Big|^2,
\end{equation*}
so we now let 
\begin{equation*}
    g(t)=t \sum_{\substack{m=1 \\ P^+(m)\leq y}}^{\infty}m e^{-\pi \frac{(mt)^2}{q} } f(m).
\end{equation*}
Here, by a similar argument $|g(t)|\ll t \cdot \big(\frac{q^{1/2}}{t}\big)^2=\frac{q}{t}$ and if $x^{3/4}\leq t\leq n$ then $|g(t+\alpha)-g(t)|\ll \alpha \cdot \frac{q}{t^2}$
so
\begin{equation*}
    \sum_{\substack{t\leq n\leq t+x^{\delta} \\ P^-(n)>y} }  |g(n)|^2\gg \frac{1}{\log y} \int_t^{t+x^{\delta} } |g(u)|^2du +O\big(\frac{x^{2\delta} q^2}{t^3}\big).
\end{equation*}
This implies
\begin{equation*}
   \sum_{\substack{x^{3/4}\leq n\leq x \\ P^{-}(n)>y }} |g(n)|^2\gg \frac{1}{\log y}\int_{x^{3/4}}^x |g(t)|^2 dt +O(q^{3/2-1/100}).
\end{equation*}
Here it is important that the error term is less than $q^{3/2}$. Letting $h(v)=g(q^{1/2}/v)$, we get 
\begin{equation*}
    \frac{1}{\log y}\int_{x^{3/4}}^x |g(t)|^2 dt= \frac{q^{1/2}}{\log y}\int_{(\log q)^{-2}}^{q^{1/8}(\log q)^{-3/2} } \frac{|h(v)|^2}{v^2}dv=\frac{q^{3/2}}{\log y}\int_{(\log q)^{-2}}^{q^{1/8}(\log q)^{-3/2} } \frac{1}{v^4} \Big|\sum_{\substack{m=1 \\ P^+(m)\leq y}}^{\infty} e^{-\pi \frac{m^2}{v^2} }m f(m)\Big|^2dv.
\end{equation*}
Here we can apply Parseval and Rankin's trick the same way.

% However, in this case we can apply Perron's formula directly with a gamma factor being in the Perron integral which makes the argument a lot simpler. We have
% \begin{equation*}
%      \sum_{n\leq x}f(n)e^{-\frac{-\pi n^2}{q}}=\sum_{n:P^+(n)\leq x } f(n)e^{-\frac{-\pi n^2}{q}}+O\big(e^{-(\log q)^3 }\big), 
% \end{equation*}
% moreover using Perron's formula and shifting the line of integration to the criticial line, we get
% \begin{equation*}
%     \sum_{n:P^+(n)\leq x } f(n)e^{-\frac{-\pi n^2}{q}}=\frac{1}{4\pi i}\int_{1/2-i\infty}^{1/2+i\infty}\Big( \sum_{n:P^+(n)\leq x}\frac{f(n)}{n^s}\Big)  \Big(\frac{q}{\pi}\Big)^{s/2} \Gamma(s/2) ds
% \end{equation*}
\section{Acknowledgements}
The author is supported by the Warwick Mathematics Institute Centre for Doctoral Training, and gratefully acknowledges funding from the University of Warwick and the UK Engineering and Physical Sciences Research Council (Grant number: EP/TS1794X/1). Part of this material is based upon work supported by the Swedish Research Council under grant no.
2021-06594 while the author was in residence at Institut Mittag-Leffler in Djursholm,
Sweden, during Spring 2024. 
\newpage
The author is grateful to Adam Harper for the many useful discussions, suggestions and for carefully reading through an earlier version of this manuscript, and would like to thank Simon Myerson for a helpful discussion during our mentor meetings. 

\section{Rights retention}
For the purpose of open access, the author has applied a
Creative Commons Attribution (CC-BY) licence to any Author Accepted Manuscript version arising from this submission.
\printbibliography
\end{document}